\documentclass[final]{amsart}

\usepackage{dbura_pdysz_PLDEforBPRE_package}

\newcounter{cnstcnt}

\newcommand{\8}{\infty}
\renewcommand{\P}{\mathbb{P}}
\newcommand{\E}{\mathbb{E}}
\newcommand{\R}{\mathbb{R}}
\newcommand{\N}{\mathbb{N}}
\newcommand{\eps}{\varepsilon}
\renewcommand{\a}{\alpha}
\renewcommand{\b}{\beta}
\renewcommand{\l}{\lambda}

\newcommand{\cC}{\mathcal{C}}

\begin{document}

\title[Large deviations for BPRE]{ Precise large deviation estimates for branching process in random environment}
	\author[D. Buraczewski, P. Dyszewski]{Dariusz Buraczewski, Piotr Dyszewski}

	\address{Instytut Matematyczny, Uniwersytet Wroclawski, Plac Grunwaldzki 2/4, 50-384 Wroclaw, Poland}
	\email{dbura@math.uni.wroc.pl, pdysz@math.uni.wroc.pl}
	\thanks{The  research was partially supported by the National Science Center, Poland (Sonata Bis, grant number DEC-2014/14/E/ST1/00588)}

\keywords{Branching process, Random environment, Large deviations, First passage time, Random walk, Central limit theorem, Law of large numbers}
\subjclass[2010]{MSC: 60J80; 60F10}

\maketitle
\begin{abstract}

We consider  a branching process in random environment $\{Z_n\}_{n\geq 0}$, which is a~population growth process
where individuals reproduce independently of each other with the 	reproduction law randomly picked at each generation.
  We describe precise asymptotics of upper large deviations, i.e. $\P[Z_n > e^{\rho n}]$. Moreover in the subcritical case, under  the Cram\'er condition on the mean
	of the reproduction law, we  investigate
	large deviation estimates for the first passage times of the branching process in question and of its total population size.

\end{abstract}


\section{Introduction}\label{sec:intro}

This work concentrates on the branching process in random environment (BPRE) introduced by Smith and Wilkinson~\cite{smith1969branching} as a  one of possible generalizations of the classical Galton-Watson process. BPRE is a population growth process
where individuals reproduce independently of each other with the 	reproduction law randomly picked at each generation.
 To put it formally, let $Q$ be a random measure on the set of non-negative integers $\NN$, that is a measurable function taking values in
$\mathcal{M} = \mathcal{M}(\NN)$ the set of all probability measures on $\NN$ equipped with the total variation distance. Then a sequence of independent identically distributed (iid) copies of $Q$, say
$\mathcal{Q}=\{ Q_n \}_{n \geq 0}$  is called a random environment. The sequence $Z=\{Z_n\}_{n \geq 0}$ is called
a~branching process in random environment $\mathcal{Q}$ if $Z_0 =1$, and
\begin{equation}\label{eq:s9}
	Z_{n+1} {=} \sum_{k=1}^{Z_n} \xi^n_k,
\end{equation}
where given $\mathcal{Q}$, $\{\xi^n_k\}_{ k \geq 0}$ are iid and independent of $Z_n$ with common distribution $Q_n$.   We would like to mention that, apart from being interesting on its own merits, as pointed out by Kesten, Kozlov and Spitzer~\cite{kesten1975limit} BPRE bears some connections to the path structure of nearest neighbour random walk in site-random environment. For a more detailed discussion regarding BPRE itself, we refer the reader to the classical book of Athreya and Ney~\cite{athreya2012branching} or a recent monograph of Kersting and
Vatutin~\cite{kersting2017discrete}.\medskip

Fundamental questions arising after introducing the process $Z$ concern its asymptotic behavior,
which  is mostly determined by the environment and
usually only  slightly depends on the detailed structure of the offspring distribution.
It turns out, that like in the case of the classical Galton-Watson process, the answer can be expressed solely in terms of the mean of the reproduction law, i.e.
\begin{equation}\label{eq: mn}
	A_k = \sum_{j=0}^\8 j Q_k(j).
\end{equation}
To be precise, denote by $\Pi_n$ the quenched expectation of $Z_n$, i.e.
$\Pi_n =	\EE[Z_n \: | \: \mathcal{Q}] = \prod_{k=0}^{n-1}A_k$
and by $A$ denote a generic copy of $A_k$.
Depending  on the behavior of the associated random walk $S_n = \log \Pi_n$ one distinguishes three cases. BPRE $Z$ is called supercritical when $S_n$ drifts to $+\infty$, subcritical when $S_n$ drifts to $-\infty$ and critical, otherwise. It is well known that if $\E[|\log A|]<\infty$, then  
$Z$ is supercritical (resp. subcritical, critical) if $\E [\log A]>0$ (resp. $\E [\log A]<0$, $\E [\log A]=0$) (Proposition 2.1 in
\cite{kersting2017discrete}).
 As a consequence it can be verified that the process $Z$ dies out  with probability 1,  
whenever $\EE[\log A] \leq 0$, whereas in the supercritical case the survival probability is positive and the population grows exponentially on the survival set (Theorems 2.1 and 2.4 in \cite{kersting2017discrete}).
 In the supercritical case the parameter  $\E[\log A]$ determines the rate of increase  of $\log Z_n$ (see Tanny \cite{tanny1} for the corresponding law of large numbers and Huang, Liu \cite{huang2012moments} for the central limit theorem).
 The relation between asymptotic behaviour of $Z$ and $\Pi_n$
goes beyond this phenomenon and was exploited in numerous papers on BPRE~\cite{afanasyev2001maximum, afanasyev2013high,bansaye:berestycki,Afanasyev:Geiger:Kersting:Vatutin:2005}.

\medskip

Our aims in this paper are twofold. First,
under mild  conditions, we describe the sharp asymptotic behaviour of large deviations of $Z$:
\begin{equation}\label{eq:intro}
\PP[Z_n > e^{\rho n}]
\end{equation}
as $n\to\8$. This problem got some attention over the past few years resulting in the precise asymptotic behaviour of~\eqref{eq:intro} only in the case of geometric reproduction law obtained by Kozlov~\cite{kozlov:2006, kozlov:2010} and on the logarithmic scale  asymptotic of~\eqref{eq:intro} which was described by Bansaye and Berestycki \cite{bansaye:berestycki}, Bansaye and B\"oinghoff \cite{Bansaye:Boinghoff:2011},  B\"oinghoff and Kersting \cite{Boinghoff:Kersting:2010}.
Recently Grama, Liu and Miqueu \cite{grama:liu:miqueu} proved precise large deviation in the sublinear regime.
In the present article we
prove that the precise asymptotic behaviour of \eqref{eq:intro} takes the form (see Theorem \ref{thm:mthm1} below)\footnote{	Here and in what follows, we  write $f(n) \sim g(n)$ for two functions $f$ and $g$ if $f(n)/g(n) \to 1$ as $n\to\8$. }
$$
	\P\big[ Z_n > e^{\rho n} \big] \sim \frac{\cC_1(\rho)}{\sqrt{n}} e^{-  I(\rho) n},
$$	
for explicitly given rate function $I(\rho)$ and for $\rho> \E[\log A]$ in the supercritical case and $\rho> 0$ in the remaining cases.
 Note that, up to a multiplicative constant, the  probability  
$\P\big[ Z_n > e^{\rho n} \big]$ possesses the same asymptotics as
$\P\big[ \Pi_n > e^{\rho n} \big]$, as seen from result of Bahadur  and  Ranga Rao \cite{bahadur:rao} or Petrov  \cite{petrov1965probabilities}. \medskip

In the second part of the paper we assume that the branching process is subcritical, i.e. $\E [\log A] < 0$. Then
 the probability of survival up to time $n$, i.e. $\P[Z_n>0]$, decays exponentially fast and the exact asymptotic behaviour was given by
Geiger,  Kersting and Vatutin \cite{geiger2003limit} (see also \cite{Birkner:Geiger:Kersting:2005}). Although the process usually dies out relatively quickly, the size of the population can be large. Thus, the above results still provide description of asymptotics of large deviations of $Z_n$. However,
further investigation of our techniques reveals that apart from knowing the probability of large deviation, we can also describe
the  exceedance time (or the first passage time) for $Z$ defined via
\begin{equation*}
	T_t^{Z} = \inf \{ n \geq 0 \: | \: Z_n>t  \}.
\end{equation*}
The first passage time $T_t^Z$ was considered in a series 
of papers by Afanasyev (see e.g. \cite{afanasyev2001maximum,afanasyev2013high,afanaseev:2014}) under the Cram\'er condition
\begin{equation}\label{eq:cramer}
	\E \big[ A^{\alpha_0} \big] = 1
\end{equation}
for some positive constant $\alpha_0$. Then as shown by Afanasyev~\cite{afanasyev2001maximum}, one has  as $t \to \infty$
\begin{equation}\label{eq:afanasyev1}
	\PP\left[ T_t^{Z} <\infty\right] = \PP\left[ \sup_{n \geq 0} Z_n > t\right] \sim c t^{-\alpha_0},
\end{equation}
for some positive constant $c$. Moreover, the asymptotic behaviour of $T_t^Z$ conditioned on $T_t^{Z} < \infty$ is
also understood to some extent, since recently Afanasyev~\cite{afanasyev2013high} proved the law of large numbers
\begin{equation}\label{eq:afanasyev2}
	\left.\frac{T_t^Z}{\log t} \: \right| \: T_{t}^Z <\infty \ \stackrel{\PP}\longrightarrow \ \frac{1}{\rho_0},
\end{equation}
and the corresponding central limit theorem \cite{afanaseev:2014}
$$
\frac{T^Z_t - \log t/ \rho_0}{\sigma_0 \rho_0^{-3/2} \sqrt{\log t}} \ \Big| \ T_t^Z <\8 \stackrel{d}\longrightarrow \mathcal{N}(0,1),
$$
where
$\rho_0 = \EE[(\log A)A^{\alpha_0}]$, $\sigma_0 = \EE[(\log A )^2 A^{\alpha_0}]$,  and $\stackrel{\PP}\longrightarrow$ (resp. $\stackrel{d}\longrightarrow$) denotes convergence in probability
(resp. in distribution).

We study the corresponding precise large deviations of the first passage time and establish asymptotics  (see Theorem \ref{thm:mthm2})
$$ \P\bigg[ \frac{T_t^Z}{\log t} < \frac 1{\rho}
\bigg] \sim \frac{\cC_2(\rho)}{\sqrt{\log t}} t^{-I(\rho)}\qquad \mbox{ for } \rho > \rho_0
$$
 and $$ \P\bigg[ \frac{T_t^Z}{\log t} > \frac 1{\rho}
\bigg] \sim \frac{\cC_3(\rho)}{\sqrt{\log t}} t^{-I(\rho)}\qquad \mbox{ for } \rho < \rho_0
$$
for some constants $\cC_2(\rho)$, $\cC_3(\rho)$. In fact we
 describe the probability that $Z$ exceeds some
threshold $t$  precisely at some given moment, that is we  show (see Theorem \ref{thm:mthm2})
\begin{equation*}
	\PP\left[T_t^{Z} = \left\lfloor \frac{\log t}{\rho} \right\rfloor\right] \sim \frac{\cC_4(\rho)}{\sqrt{\log t }} t^{- I(\rho)},
\end{equation*}
for some  constant $\cC_4(\rho)$.  As one may expect
from~\eqref{eq:afanasyev1} and \eqref{eq:afanasyev2}, $I(\rho)$ attains its smallest
value at $\rho_0$ which is
$I(\rho_0) = \alpha_0$.

 \medskip

 The next  process of our interest is $W=\{W_n\}_{n\geq 0}$ describing the total population size up to time $n$:
$$
	W_n = \sum_{k=0}^n Z_k.
$$
In Theorem \ref{thm:mthmld} we state the  large deviations results for $W_n$.
  We also investigate the corresponding first passage time, that is
\begin{equation*}
	T_t^{W} = \inf \{ n \geq 0 \: | \: W_n>t  \}
\end{equation*}
as $t \to \infty$. Then, by the arguments presented by Kesten et al.~\cite{kesten1975limit} and Afanasyev~\cite{afanasyev2001maximum}
\begin{equation}
\label{eq:af-kes}
	\PP\left[ T_t^{W} <\infty\right] = \PP\left[ \sum_{n = 0}^{\infty} Z_n > t\right] \sim c t^{-\alpha_0}.
\end{equation}
After comparing~\eqref{eq:afanasyev1} and~\eqref{eq:af-kes} one may expect that $T_t^W$ possesses the same normalisation for the law of large numbers and central limit theorem as $T_t^Z$. This is indeed the case and in  Theorem~\ref{thm:mthmclt} we establish the conditional limit theorems for $T_t^W$
as well as a lower deviations result.

\medskip

The paper is organized  as follows. In Section~\ref{sec:results} we present  a precise statements of our results.
In Section~\ref{sec:prel} we provide preliminaries used in all the proofs for the branching process $Z$  and we present proof of Theorem \ref{thm:mthm1}.
In Sections~\ref{sec:bp} and \ref{sec:total} we prove our main results for the subcritical case.
We will denote the constants by $c_i$, $i \in \N$ and if a constant is not of our interest, it will be denoted by a generic $c$, which may change from one line to the other. Constants appearing in our claims will be denoted by $\cC_i$ with $i \in \N$.

\section{Main Results}\label{sec:results}
\subsection{Definitions and assumptions}
We denote by $P_0$ a given law  on $\mathcal M$, the set of all probability measures on $\N$.  Then the probability measure  $P = P_0^{\otimes \N}$ on ${\mathcal M}^{\otimes \N}$ defines the law of the environment ${\mathcal Q}$.
Let $(\Gamma, \mathcal{G}) = (\NN^\NN, \mathcal{B}or(\NN^\NN))$ be the canonical measurable space on which, given the environment ${\mathcal Q}$,  the BPRE $Z$ is defined and denote by $\PP_{\mathcal{Q}}$ the corresponding probability measure.
Then $(\Gamma \times {\mathcal M}^{\otimes \N}, \P)$, for $\P = \P_{\mathcal Q}\otimes P$, is the total probability space considered below. We will occasionally write $\PP[ \cdot \: | \: \mathcal{Q}] = \PP_{\mathcal{Q}}[\cdot]$.

As mentioned in the previous section, the multiplicative random walk $\{\Pi_n\}_{n \geq 0}$ plays a crucial role in our analysis, for
this reason denote
$$
	\lambda(\alpha) = \EE[A^{\alpha}]\quad \mbox{ and } \quad
	\Lambda(\alpha) = \log \lambda(\alpha) = \log \EE[A^{\alpha}].
$$
Put   $\alpha_{\infty} = \sup\{ \alpha>0 \: | \: \lambda(\alpha) <\infty \}$ and 	$\alpha_{min} =  {\rm arg\,min}_{\alpha\ge 0}\: \lambda(\alpha)$. Then the domain of $\lambda$ and $\Lambda$ is $[0,\a_\8)$
and  both functions are smooth and convex in $(0, \alpha_\infty)$.
Denote by  $\rho(\alpha) = \Lambda'(\alpha)$ and  			
	$\sigma(\alpha) = \sqrt{\Lambda ''(\alpha)}$ standard parameters related to the function $\Lambda$.
Let $\rho_\8 = \sup\{\rho(\a)| \a<\a_\8\}$.
 Recall that the convex conjugate (or the Fenchel-Legendre transform) of $\Lambda$ is defined by
$
	\Lambda^*(x) = \sup_{s\in \R}\{sx - \Lambda(s)\}$, for  $x\in\R$.
This rate function appears in the study of large deviations problems for random walks (see e.g.
Dembo   and   Zeitouni \cite{DZ}).
 A straightforward calculation yields $ \Lambda^*(\rho) = \alpha \rho - \Lambda(\alpha)$ for $\rho = \rho(\alpha)$.

Lastly,   we  assume throughout the paper that
\begin{equation}\label{eq:stand25}
\mbox{$A>0$	a.s. and the law of  $\log A$ is non-lattice,}
\end{equation}
i.e. $\log A$ is not supported on any of the sets $a{\mathbb Z}+b$ for $a>0$ and $b\in \R$.

\subsection{Large deviations of $Z$}
First we state large deviations principle for the BPRE $Z$.

\begin{thm}\label{thm:mthm1}
 Assume that \eqref{eq:stand25} holds and fix $\alpha \in (0,\alpha_\infty)$ such that $\rho = \rho(\alpha)>0$. Assume moreover that
\begin{itemize}
  \item[(H1)] if $\alpha < 1$,  then $\E [A_0^{\alpha-1}Z_1 \log^+ Z_1]<\infty$ and
  \begin{equation}\label{eq:nd1}
  \E\Big[ \E\big[ Z_1^{\frac{\a}{\a -\delta_\a}}\big| {\mathcal Q}\big]^{{\alpha-\delta_\a}}\Big] <\infty
  \end{equation} for some $\delta_\a \in (0,\alpha)$.
  \item[(H2)]  if $\alpha = 1$, then $\E [Z_1^{1+\delta_1}] <\infty$ for some $\delta_1>0$;
  \item[(H3)]  if $\alpha > 1$, then $\lambda(\alpha)>\lambda(1)$ and $\E [Z_1^{\alpha+\delta_\a}]<\infty$ for some $\delta_\a>0$.
\end{itemize}
 Then
	$$
		\P\big[ Z_n  >  e^{\rho n} \big] \sim \frac{\cC_1(\rho)}{\sqrt n}  e^{-n\Lambda^*(\rho)}
\qquad \mbox{as } n\to\8,
	$$
	where the constant $\cC_1(\rho)$ can be represented by
	\begin{equation}\label{eq:s12}
		\cC_1(\rho) = \frac{1}{\alpha\sigma(\alpha)\sqrt{2 \pi}} \lim_{k \to \infty} \frac{\EE Z_k^{\alpha}}{\lambda(\alpha)^k}
	\end{equation}
	and the limit on the right-hand side exists, it is finite and strictly positive.
\end{thm}
This result says that up to a constant, in view of the Bahadur, Ranga Rao and Petrov results (see Lemma \ref{lem: petrov1} below), large deviations of $Z_n$ and $\Pi_n$ coincide.
Note that, since the function $\lambda$ is convex, Theorem  \ref{thm:mthm1} is valid for arbitrary $\alpha < \alpha_\infty$ both in the supercritical and critical case (of course when the required moment conditions are satisfied). Moreover,
if $\alpha \le  \min\{1, \alpha_\infty\}$ and $\rho(\alpha)>0$, both conditions (H1) and (H2) are fulfilled in the subcritical case.
The assumption
$\lambda(\alpha)>\lambda(1)$ for $\alpha>1$  can fail only in the strongly subcritical
case, when $\lambda'(1) = \E[A \log A]<0$.  Then Kozlov \cite{kozlov:2006, kozlov:2010} and B\"oinghoff, Kersting \cite{Boinghoff:Kersting:2010} observed that
the asymptotics of $\P[Z_n > e^{\rho n}]$ and  $\P[\Pi_n > e^{\rho n}]$ differ for small values of $\rho$. In this case also our argument breaks down, see Remark \ref{rem:s1} for more detailed explanations. Nevertheless even in the strongly subcritical case for large parameters $\alpha$ satisfying (H3) the above Theorem is valid.

Finally note that condition \eqref{eq:nd1} for $\a<1$, by the Jensen inequality, is guaranteed by a stronger hypothesis 
$\E[ Z^{1+\delta'_{\lambda}}]$ for $1+\delta'_{\lambda} = \alpha/(\a - \delta_{\lambda})$, which reminds the corresponding assumptions for $\a\ge 1$. 

\subsection{Large deviations of $T_t^Z$}
 Now assume that the BPRE $Z$ is subcritical, i.e.
\begin{equation}\label{eq:h1}
  \E \log A < 0,
\end{equation}
which means the process $Z$ dies out a.s, that is $\PP[\lim_{n\to \infty} Z_n=0]=1$. Take $n$ to depend on $t$ via
\begin{equation}\label{eq:nt}
	n = n(t) = n_t = \left\lfloor \frac{\log t}{\rho}  \right\rfloor
\end{equation}
and denote additionally
\begin{equation*}	
	\Theta (t) = \frac{\log t}{\rho} - \left\lfloor \frac{\log t}{\rho}  \right\rfloor.
\end{equation*}
Under this relation we are able to provide exact asymptotic of $\PP\left[T_t^{Z} = n \right]$. Recall that $\rho_0 = \E[A^{\alpha_0}\log A] = \lambda'(\alpha_0)$.

\begin{thm}\label{thm:mthm2} Let the assumptions of Theorem~\ref{thm:mthm1} be in force.
Assume  \eqref{eq:h1} and  that
	$n$ and $t$ are related via~\eqref{eq:nt}.
Then, there are constants $\cC_2(\rho)$ and $\cC_3(\rho)$ such that
 for $\rho > \rho_0$
	\begin{equation}\label{eq:2.4down}
			\PP\left[ n\le T_t^{Z} <\infty \right] \sim
			\cC_2(\rho) \lambda(\alpha)^{-\Theta (t)} \;
			\frac{t^{-\Lambda^*(\rho)/\rho}}{\sqrt{\log t}} \qquad \mbox{ as } t\to\8
	\end{equation}
and for $\rho<\rho_0$
	\begin{equation}\label{eq:2.4up}
			\PP\left[T_t^{Z} \ge n  \right] \sim
			\cC_3(\rho) \lambda(\alpha)^{-\Theta (t)} \;
			\frac{t^{-\Lambda^*(\rho)/\rho}}{\sqrt{\log t}} \qquad \mbox{ as } t\to\8.
	\end{equation}
Moreover for some constant $\cC_4(\rho)$
	\begin{equation}\label{eq:mthm2point}
			\PP\left[T_t^{Z} = n  \right] \sim
			\cC_4(\rho) \lambda(\alpha)^{-\Theta (t)} \;
			\frac{t^{-\Lambda^*(\rho)/\rho}}{\sqrt{\log t}} \qquad \mbox{ as } t\to\8.
	\end{equation}
\end{thm}

Up to our best knowledge precise large deviations of $T_t^Z$ of the form~\eqref{eq:mthm2point} were not studied in the literature.
Result in the same vein but for the sequence of products $\{\Pi_n\}_{n \geq 0}$ was recently obtained by Buraczewski and Ma\'slanka~\cite{buraczewski2016precise} incorporating techniques used previously in work of Buraczewski,
Damek and Zienkiewicz~\cite{buraczewski2015pointwise} in the context of perpetuities (see \eqref{eq:wt1} for an example  of a perpetuity). In our case large  deviations for $Z$ are caused by deviations
for the environment, that is the multiplicative random walk $\{ \Pi_n \}_{n \geq 0}$.

\subsection{Large deviations of $W$ and $T_t^W$}
Turning our attention to the total population size, we will approximate $W_n$ by its conditional mean, that is
\begin{equation}\label{eq:wt1}
	R_{n}=\EE[W_n \: | \: \mathcal{Q}] = \sum_{k=0}^n\prod_{j=0}^{k-1}A_j = \sum_{k=0}^n \Pi_{k}.
\end{equation}
Note that $ \{ R_n \}_{n \geq 0}$ forms so-called perpetuity sequence and that its structure is more complicated than the one of $\{\Pi_n\}_{n \geq 0}$. Working with perpetuities requires usually more advanced techniques and sometimes this process reveals some new properties (we refer to \cite{buraczewski:damek:mikosch} for more details). Nevertheless, in many aspects the asymptotic behaviour
of $\{R_n\}_{n \geq 0}$ is similar to the one of $\{\Pi_n\}_{n \geq 0}$. Thus our main results concerning the total population resemble those stated above, but are slightly weaker (in these settings we were not able to provide sharp pointwise estimates as in \eqref{eq:mthm2point}).

We assume below the Cram\'er condition \eqref{eq:cramer}  and formulate large deviations for $\a>\a_0$ (Theorem \ref{thm:mthmld}),
 the law of large numbers  and the central limit theorem (Theorem \ref{thm:mthmclt}).

\begin{thm}\label{thm:mthmld}
	Under the assumptions of Theorem~\ref{thm:mthm1} fix $\rho = \rho(\a)$ for some $\a\in (\a_0,\a_\8)$ and assume additionally \eqref{eq:h1}. Then, there exists a constant
	$\cC_5(\rho)$ such that
	$$
		\P\big[ W_n > e^{\rho n} \big]\sim \frac{\cC_5(\rho)}{\sqrt n} e^{-n\Lambda^*(\rho) }\quad \mbox{as } n\to\8
	$$
	and
	$$
		\P\big[ W_{\lfloor n- D\log n \rfloor}  > e^{\rho n} \big]=  o\bigg(\frac{1}{\sqrt n} e^{-n\Lambda^*(\rho) } \bigg)\quad
		\mbox{as } n\to\8
	$$
	for any $D > \Lambda(\a)^{-1} (2\a +7)$. In particular if $n$ and $t$ are related via \eqref{eq:nt}, then
	\begin{equation*}
		\PP\left[T_t^{W} \leq n \right] \sim \cC_5(\rho) \lambda(\alpha)^{-\Theta(t)} \frac{t^{-   \Lambda^*(\rho)/\rho  
}}{\sqrt{\log t}},
	\end{equation*}
	and		
	\begin{equation*}
		\PP\left[T_t^{W} \leq n  - D \log n \right] =o \bigg( \frac{t^{- \Lambda^*(\rho)/\rho
}}{\sqrt{\log t} } \bigg).
	\end{equation*}
\end{thm}

From our approach and results stated so far, we know how the deviations of $Z$ and $W$ can occur and more importantly, what is
the most probable moment of such deviation. With that knowledge, we are able to derive the corresponding law of large numbers and central limit theorem for $T_t^W$.

\begin{thm}\label{thm:mthmclt}
	Suppose the assumptions of Theorem~\ref{thm:mthm1} are in force. Fix $\rho = \rho(\a)$ for some $\a\in (\a_0,\a_\8)$ and assume additionally \eqref{eq:h1}. Then
	$$
		\frac{T_t^W}{\log t}\ \Big| \ T_t^W <\8 \stackrel{\PP}\longrightarrow\frac 1{\rho_0}
	$$
	and
	$$
		\frac{T_t^W - \log t/\rho_0}{\sigma_0 \rho_0^{-\frac 32}\sqrt{\log t}}\ \Big| \ T_t^W <\8
		\stackrel{d}\longrightarrow \mathcal{N}(0,1),
	$$
	where $\rho_0 = \rho(\alpha_0)$ and $\sigma_0 = \sigma(\alpha_0)$. Moreover
	$$
		\P\big( W_n > e^{\rho_0 n}\big) \sim \cC_6(\rho_0) e^{-\a_0 \rho_0 n}.
	$$
\end{thm}

Observe that the result in Theorem \ref{thm:mthmld} is weaker than those
in Theorems \ref{thm:mthm1} and \ref{thm:mthm2}. However similar situation holds when we compare the results concerning $\{\Pi_n\}_{n \geq 0}$ and perpetuities $\{R_n\}_{n \geq 0}$. Then for $\a<\a_0$ the asymptotic behaviour of the deviation of both processes can be  different, see \cite{buraczewski2014large, buraczewski2015pointwise}.

The same techniques can be used to study BPRE with immigration having applications to random walk in random environment. Similar  scheme allows to describe precise  large deviations for random walks in random environment \cite{buraczewski:dyszewski}.

\section{Auxiliary results and proof of Theorem \ref{thm:mthm1}}\label{sec:prel}
\subsection{Moments of $Z_n$}
We start with a description of asymptotic behavior
of  moments of $Z_n$, as
$n \to \infty$.  The following lemma guarantees that in terms of moment generating functions both processes $Z_n$ and $\Pi_n$ are
asymptotically equivalent and  proves in particular existence of the limit in \eqref{eq:s12}.
\begin{lem}
\label{lem:s1}
If hypotheses of Theorem \ref{thm:mthm1} are satisfied, then the limit
	\begin{equation}\label{eq:momzn}
		c_Z=c_Z(\alpha)= \lim_{n \to \infty}\frac{\EE[Z_n^{\alpha}]}{\lambda(\alpha)^{n}}
	\end{equation}
	exists and $c_Z(\a) \in (0,\infty)$.
\end{lem}
\begin{proof}
If $\alpha=1$, the lemma follows trivially since $\E[Z_n] = \lambda(1)^n$. For $\alpha \not= 1$
the lemma was essentially proved by Huang and Liu (Theorem 1.3 \cite{huang2012moments}). However, they worked  under  different   assumptions (the paper was written only for the supercritical case and the authors assumed e.g. $\P[Z_1=0]=0$). For reader convenience we present below complete argument following \cite{huang2012moments}.

\medskip

Let $U_n = Z_n/\Pi_n$ be the normalized population size. Then the sequence $\{U_n\}_{n\ge 0}$ forms a~martingale both under the quenched probability $\P_{\mathcal{Q}}$ and under the annealed law $\P$ with respect to the filtration ${\mathcal F}_n = \sigma(
Q_j, \xi_k^j, k \in \NN, \: j\le n)$. As a positive martingale $U_n$ converges a.s. to some random variable $U$
such that $\E U\le 1$ by an appeal to Fatou's Lemma. Of course in the critical and the subcritical case, since the population goes extinct, $U = 0$ a.s.  In the supercritical case Tanny \cite{tanny1988necessary} proved that $U$ is nondegenerate if and only if
\begin{equation}\label{eq:s3}
  \E[A_0^{-1} Z_1 \log^+ Z_1] < \infty.
\end{equation}
Since $\lambda(\alpha)<\infty$, we can define a new probability measure on $\mathcal M(\NN)$ viz.
$$
 P_{0,\alpha}(dQ) = \frac{A^{\alpha}P_0(dQ)}{\lambda(\alpha)}, \qquad 	A = \sum_{j=0}^\8 j Q(j).
$$
We consider the new BPRE with respect to the environment distributed according to $P_{\alpha} = P_{0,\alpha}^{\otimes \N}$ and we define
on $(\Gamma \times {\mathcal M}^{\otimes \N} )$
the probability measure
$\P_{\alpha} = \P_{\mathcal{Q}}\otimes P_{\alpha}$. Let $\E_{\alpha}$  denote the corresponding expectation. Observe that
\begin{equation}\label{eq:wt3}
\E_{\alpha}[ U_n^{\alpha}] = \frac{\E [Z_n^{\alpha}]}{\lambda(\alpha)^n}.
\end{equation}
 Thus to ensure \eqref{eq:momzn} it is sufficient to prove
existence of the limit $\lim_{n\to\infty} \E_{\alpha} [ U_n^{\alpha}]$.
Notice that
$$
\E_{\alpha}[ \log A] = \frac{\E [A^{\alpha} \log A]}{\lambda(\alpha)} = \Lambda'(\alpha) = \rho(\alpha) > 0.
$$
Therefore, under $\P_\alpha$ the process $\{Z_n\}_{n\ge 0}$ is supercritical.

\medskip

\noindent
{\sc Case 1. Assume $\alpha <1$.}
Since by (H1)
$$
\E_{\alpha}\big[ A_0^{-1} Z_1 \log^+ {Z_1} \big] = \frac 1{\lambda(\alpha)} \E \big[ A_0^{\alpha-1} Z_1 \log^+ Z_1 \big] <\infty,
$$   appealing to \eqref{eq:s3} and Tanny \cite{tanny1988necessary} we deduce that $\P_{\alpha}[U>0] >0$  and
$U_n$ converges to $U$ in $L^1(d\P_{\alpha})$. Finally
$$
\lim_{n\to\infty} \frac{\E [Z_n^{\alpha}]}{\lambda(\alpha)^n} = \lim_{n\to\infty} \E_{\alpha} [U_n^{\alpha}] = \E_{\alpha}[U^{\alpha} ] >0.
$$

\medskip

\noindent
{\sc Case 2. Assume $\alpha > 1$.} Now we are going to appeal to Guivarc'h and Liu  \cite{guivarc2001proprietes} (Theorem 3), who proved that
\begin{equation}\label{eq:wt2}
0 < \E_{\alpha}U^{\alpha} <\infty \quad \mbox{if and only if} \quad \E_{\alpha}\big[ \big( Z_1/A_0 \big)^\alpha\big] <\infty\ \mbox{ and } \E_{\alpha} [A_0^{1-\alpha}] < 1
\end{equation}
and both statements are equivalent to $U_n \to U$ in $L^{ \alpha}(d\PP_\alpha)$. Since by (H3)
$$
\E_{\alpha}\big[\big( Z_1/A_0 \big)^\alpha\big]  =  \lambda(\alpha)^{-1} \E [Z_1^{\alpha}] <\infty
$$
and
\begin{equation}\label{eq:s45}
\E_{\alpha} [A_0^{1-\alpha}] = \lambda(1)/\lambda(\alpha) < 1,
\end{equation}
the martingale $U_n$ converges to $U$ in $L^\alpha(d\P_{\alpha})$ and combining \eqref{eq:wt3} with \eqref{eq:wt2} we conclude the Lemma.
\end{proof}
\begin{rem} \label{rem:s1}
The above proof explains the role of the hypothesis $\lambda(1) < \lambda(\alpha)$ in the strongly subcritical case, which appears in \eqref{eq:s45}. If this condition is not satisfied in view of the Guivarc'h, Liu result  \cite{guivarc2001proprietes} we are not able to ensure finiteness of $\E_\alpha U^{\alpha}$
and prove \eqref{eq:momzn}.
\end{rem}

To prove our main results we need a stronger version of Lemma \ref{lem:s1}. Below we prove
 locally uniform estimates for moments of $Z_n$ and an auxiliary inequality \eqref{eq:s5}.

\begin{lem}
\label{lem:s4}
 Under  hypotheses  of Theorem \ref{thm:mthm1}  there are $\delta>0$ and $c = c(\alpha,\delta)$ such that
\begin{equation}\label{eq:s4}
\E[Z_n^s] \le \left\{ \begin{array}{cl}  c \lambda(s)^n \qquad \mbox{ for any } s\in [\alpha,\alpha+\delta]  & \mbox{if }\alpha \neq 1, \\
					 cn^8 \lambda(s)^n \qquad \mbox{ for any } s\in [1,1+\delta]  &  \mbox{if } \alpha=1. \end{array}\right.
\end{equation}
Moreover there exists $\gamma<\lambda(\alpha)$ and some constant $c$ such that for any $n\ge 1$
\begin{equation}
\label{eq:s5}
 \E\big[ |Z_n - A_{n-1} Z_{n-1}|^{\alpha}  \big] \le c \gamma^n.
\end{equation}
\end{lem}
\begin{proof}
  \noindent {\sc Step 1. Proof of \eqref{eq:s4} and \eqref{eq:s5} for $\alpha <1$.}
 Take any $s\in [\alpha,1]$. Since the function $x\mapsto x^s$ is concave,	the conditional Jensen inequality entails
\begin{equation}\label{eq:sr4}
\EE[Z_n^s] =  \E \big[ \E[Z_n^s|\mathcal{Q}]\big]
\le \E[ \Pi_n^s] = \lambda(s)^{n},
\end{equation}
which proves  \eqref{eq:s4} with $c=1$ and any $\delta \le 1 - \alpha$.

The second part of the Lemma for $\a<1$ follows essentially from the Marcinkiewicz-Zygmund inequality saying that if
$\{X_i\}_{i\in \N}$ is a sequence of i.i.d. random variables such that $\E X = 0$, $\E|X|^p<\infty$ for some $p\in[1,2]$,
$N$ is a stopping time and $S_n = \sum_{i=1}^n X_i$, then
\begin{equation}\label{eq:sr2}
\E |S_N|^p \le c_p \E|X_1|^p \cdot \E N
 \end{equation} for some universal constant $c_p$ depending only on $p$ (see e.g. Theorem I.5.1 in \cite{gut2009stopped}). We also need
the following decomposition being an immediate consequence of  \eqref{eq:s9}:
	\begin{equation}\label{eq:1}
		Z_n - A_{n-1} Z_{n-1} = \sum_{k=1}^{Z_{n-1}}(\xi^{n-1}_k - A_{n-1}) ,
	\end{equation}
	where $\xi^{n-1}_k - A_{n-1}$, given $\mathcal{Q}$, are iid with zero mean and independent of $Z_{n-1}$.
\newline
 Since $\rho(\alpha)>0$, there exists  $\eps \le \min\{\alpha/2, \delta_\a \}$ for $\delta_\a$ as in condition (H1), such that
$\lambda(\alpha-\eps)< \lambda(\alpha)$.
Applying the conditional Jensen inequality  for the convex function $x\mapsto x^{\frac 1{\alpha - \eps }}$ and next
the Marcinkiewicz-Zygmund inequality \eqref{eq:sr2} with $p = \frac{\a}{\a-\eps}$ 
(thanks to our choice of $\eps$ we have $\frac{\alpha}{\alpha-\eps} < 2$) we can write
\begin{equation}\label{eq:sr5}
\begin{split}
  \E\big[|Z_n - A_{n-1} Z_{n-1}|^{\alpha}\big] &\le  \E\Big[ \E \big[ |Z_n - A_{n-1} Z_{n-1}|^{\frac{\alpha}{\alpha-\eps}} | \mathcal{Q}, \: Z_{n-1}  \big]^{\alpha-\eps} \Big]\\
   &\le c  \E\bigg[ \E\bigg[ \bigg| \sum_{k=1}^{Z_{n-1}}(\xi^{n-1}_k - A_{n-1}) \bigg|^{\frac{\alpha}{\alpha-\eps}} \Big| \mathcal{Q}, \:Z_{n-1} \bigg]^{\alpha-\eps}
   \bigg]\\
&\le c  \E\Big[  Z_{n-1}^{\alpha - \eps}  \E\big[|\xi^{n-1}_1-A_{n-1}|^{\frac{\alpha}{\alpha-\eps}} | \mathcal Q\big]^{\alpha-\eps} \Big]\\
&\le c  \E\big[  Z_{n-1}^{\alpha - \eps}\big] \cdot
  \E\Big[ \E\big[ |\xi_1^{n-1}- A_{n-1}|^{\frac{\a}{\a -\eps}}\big| {\mathcal Q}\big]^{{\alpha-\eps}}\Big]
\end{split}
\end{equation}
where the last inequality follows from independence od $Z_{n-1}$ and the pair $(\xi_1^{n-1},A_{n-1})$. The last expression is bounded by hypothesis (H1), because $(\xi_1^{n-1},A_{n-1})$ has the same distribution as $(Z_1,A_0)$. Finally appealing to \eqref{eq:sr4} we conclude \eqref{eq:s5}.

\medskip

\noindent {\sc Step 2: Proof of \eqref{eq:s4}  for $\alpha \in [1, \alpha_{\infty})$.}
 Below we will prove a slightly different inequality, saying
 that for any   $\beta\in[1,\alpha +\delta_\a)$ ($\delta_\a$ was defined in (H2) and (H3))
\begin{equation}\label{eq:s30}
  \E[Z_n^{\beta}] \le d(\beta)\cdot\; \left\{
  \begin{array}{c@{\ \ \mbox{ if }}l}
 \lambda(\beta)^n, &   \lambda(\beta)>\lambda(1),\\
 n^{6\beta+1}\lambda(1)^n, &   \lambda(\beta)\leq\lambda(1),\\
  \end{array}
  \right.
\end{equation}
for some continuous function $d$ on $[1,\alpha+\delta_\a)$. 

 Note that \eqref{eq:s30} entails \eqref{eq:s4}. Indeed, this is clear for $\alpha > 1$ (it is sufficient to choose $s = \sup_{\b\in[\a,\a+\delta)} d(\beta)$). Whereas for $\a=1$, just recall that
$\rho = \rho(1)>0$, thus $\lambda(s)>\lambda(1)$ for $s>1$.

\medskip

For any $\varepsilon, \beta >0$  and $x,y> 0$, the following inequality holds
	$$|x+y|^{\beta} \leq (1+\varepsilon)|x|^{\beta} + c(\beta,\varepsilon)|y|^{\beta},$$
	where $c(\beta,\varepsilon) = (1-(1+\varepsilon)^{-\beta^{-1}})^{-\beta}\sim (\varepsilon/\beta)^{-\beta}$ as $\varepsilon\to 0$.
	One can easily verify it by considering two cases $|y| \lesseqgtr ((1+\varepsilon)^{1/\beta}-1)|x|$. Combining the above inequality with
decomposition	$Z_n {=}Z_{n-1}A_{n-1}+ ( Z_n - Z_{n-1}A_{n-1})$ we obtain
that for any
	$\varepsilon > 0$
	\begin{equation}\label{eq:s23}
		\EE[Z_n^{\beta}] \leq (1+\varepsilon) \lambda(\beta) \EE[Z_{n-1}^{\beta}] +
			c(\beta,\varepsilon) \EE\big[ \big| Z_n - A_{n-1}Z_{n-1} \big|^{\beta}\big].
	\end{equation}
Proceeding as in \eqref{eq:sr5}, by
 the Marcinkiewicz-Zygmund inequality \eqref{eq:sr2} and the decomposition \eqref{eq:1},  we obtain 
	\begin{equation}\label{eq:s15}
\E\big[| Z_n-A_{n-1}Z_{n-1}|^{\beta}\big] \le \E\big[|Z_1 - A_0|^{\beta}\big] \E[Z_{n-1}^{\beta^*}],
	\end{equation}
where $\beta^* = \frac{\beta}{2}\vee 1$,  see Theorem I.5.1 in \cite{gut2009stopped} for details.
(Note that we are not able here to deduce here \eqref{eq:s5} because in the subcritical case it can happen that $\rho(\beta^*)< 0$ and thus we cannot refer to Lemma   \ref{lem:s1} and bound $\E[Z_{n-1}^{\beta^*}]$ by $c \lambda(\beta^*)^n$.)

 Combining \eqref{eq:s23},  \eqref{eq:s15} with the inequality $(1+\varepsilon)\leq e^\varepsilon$ for $\eps = n^{-2}$ we  end up with the following inequality
	\begin{equation*}
		\EE[Z_n^{\beta}] \leq e^{\frac 1{n^2}}\lambda(\beta)\EE[Z_{n-1}^{\beta}] +
			c(\beta)  n^{2\beta} \EE[Z_{n-1}^{\beta^*}],
	\end{equation*}
where $c(\beta) = c_0 \E\big[ |Z_1-A_0|^{\beta} \big] \beta^\beta$.
	By iterating this inequality we obtain
	\begin{equation}\label{eq:3.4}
\begin{split}
		\EE[Z_n^{\beta}] &\leq e^{\frac 1{n^2}} \lambda(\beta)\bigg( 	 e^{\frac 1{(n-1)^2}}\lambda(\beta)\EE[Z_{n-2}^{\beta}] + c(\beta) (n-1)^{2\beta} \EE[Z_{n-2}^{\beta^*}] \bigg) + c(\beta) n^{2\beta} \EE[Z_{n-1}^{\beta^*}]\\
			&\le e^{1+\frac 1{2^2}+\cdots+ \frac{1}{n^2}} \lambda(\beta)^n + c(\beta)\sum_{k=1}^n e^{\frac{1}{(n-k+1)^2}+\cdots+ \frac{1}{n^2}} \lambda(\beta)^{k-1} (n-k+1)^{2\beta} \EE[Z_{n-k}^{\beta^*}]\\
			&\le e^2 \lambda(\beta)^n + e^2 c(\beta)\sum_{k=1}^n  \lambda(\beta)^k (n-k+1)^{2\beta} \EE[Z_{n-k}^{\beta^*}].
\end{split}	
\end{equation}
 The above inequality is the key step in the proof of \eqref{eq:s30}. From now we proceed by induction on $m$ such that $\beta\in (2^m,2^{m+1}]$, i.e. our induction hypothesis is \eqref{eq:s30}.

Assume first that $\beta\in (1,2]$, then $\beta^*=1$ and $\E[Z_{n-k}^{\beta^*}] =\E[Z_{n-k}] = \lambda(1)^{n-k}$. If
$\lambda(\beta)>\lambda(1)$, then
$$
\E[Z_n^{\beta}] \le d(\beta) \lambda(\beta)^n
$$
for
$$
d(\beta) = e^2\bigg[ 1+ c(\beta)\sum_{k=0}^{\infty} \bigg( \frac{\lambda(1)}{\lambda(\beta)} \bigg)^k  (k+1)^{2\beta} \bigg].
$$
Otherwise, if $\lambda(\beta)\leq  \lambda(1)$,
$$
\E[Z_n^{\beta}] \le d(\beta) n^{2\beta+1} \lambda(1)^n
$$
for
$d_2(\beta) = e^2\big( 1+c(\beta)\big)$. Therefore we obtain \eqref{eq:s30} for $\beta \in (1,2]$.

Assume now that $\beta\in (2^m,2^{m+1}]$ for $m>1$. We again consider two cases.
If $\lambda(\beta)>\lambda(1)$, then $\lambda(\beta/2)< \lambda(\beta)$ and by the induction hypothesis
\begin{equation}\label{eq:wt5}
  \E\big[ Z_n^{\beta/2} \big] \le d(\b/2) n^{6\beta+1} \max\{ \lambda(1), \lambda(\beta/2) \}^n.
\end{equation}
Since $\beta^* = \beta/2$, combining the above inequality with \eqref{eq:3.4} we obtain
$$
\E \big[ Z_n^\beta \big] \le d(\beta) \lambda(\beta)^n
$$ for
$$d(\beta) = e^2\bigg[ 1+ c(\beta)\sum_{k=0}^{\infty} d(\b/2) k^{3\beta+1} \bigg( \frac{
 \max\{ \lambda(1), \lambda(\beta/2) \}
}{\lambda(\beta)} \bigg)^k  (k+1)^{2\beta} \bigg].
$$ Finally, assume $\lambda(\beta)\le\lambda(1)$. In this case $\lambda(\beta/2)< \lambda(1)$ and by the induction hypothesis
\begin{equation}\label{eq:wt6}
  \E\big[ Z_n^{\beta/2} \big] \le d(\b/2) n^{3\beta+1} \lambda(1)^n.
\end{equation}
Therefore, in view of \eqref{eq:3.4} we have
$$
\E \big[ Z_n^\beta \big] \le d(\beta) n^{6\beta + 1}  \lambda(1)^n
$$ for $d(\beta) = e^2 (1+ c(\beta))$.

\medskip

\noindent {\sc Step 3.  Proof of  \eqref{eq:s5}  for $\alpha\in [1,\alpha_{\infty}]$.} Observe that for $\alpha>1$ combining \eqref{eq:s30}
with \eqref{eq:s15} we obtain \eqref{eq:s5}.  Indeed, since $\rho(\alpha)>0$ one can take $\gamma = \lambda(\a_1)$ for some $\a_1 < \a$
such that $\max\{ \lambda(1), \lambda(\alpha^*)  \} < \lambda(\a_1) <  \lambda(\a) $, where $\a^* = \a/2\vee 1$.

It remains to prove \eqref{eq:s5} for $\alpha=1$, which can be done applying similar arguments as above. Namely, choose $\eps<\delta_1$
such that $\lambda(1/(1+\eps)) <\lambda(1)$. Then appealing to \eqref{eq:1} and  the Marcinkiewicz-Zygmund inequality  \eqref{eq:sr2}
 we obtain
$$
\E\big[ |Z_n-A_{n-1} Z_{n-1}|^{1+\eps} | {\mathcal Q} \big] \le
\E\big[ Z_{n-1} | {\mathcal Q} \big] \cdot
\E\big[ |\xi_1^{n-1} -A_{n-1} |^{1+\eps} | {\mathcal Q} \big].
$$
 Hence, the conditional Jensen inequality entails
\begin{align*}
    \E\big[|Z_n - A_{n-1} Z_{n-1}|\big] &\le
  \E\Big[ \E \big[ |Z_n - A_{n-1} Z_{n-1}|^{1+\eps} | \mathcal{Q}  \big]^{\frac{1}{1+\eps}} \Big] \\
  &\le c \E\Big[ \Pi_{n-1}^{1/(\alpha+\eps)}  \E\big[  |\xi_1^n|^{1+\eps} + |A_{n-1}|^{1+\eps}  | {\mathcal Q}  \big] \Big]\\
  &\le c \big( \E [Z_1^{1+\eps}]  +\lambda(\alpha+\eps) \big) \cdot \lambda(1/(1+\eps))^n.
\end{align*}
We conclude \eqref{eq:s5} with $\gamma = \lambda(1/(1+\eps))$ and complete thus  proof of the Lemma.
\end{proof}

\subsection{Large deviations of $\{\Pi_n\}_{n \geq 0}$ and its perturbations}

Now we recall classical results concerning large deviations for random walks due to Bahadur and Ranga Rao \cite{bahadur:rao} and Petrov \cite{petrov1965probabilities}.

\begin{lem}\label{lem: petrov1}
Fix an arbitrary  $\a<\a_\infty$ such that $\rho = \rho(\a)>0$. 
	If $\{\overline{\delta}_n\}_{n \geq 0}$ is a sequence converging to $0$, then for any $\varepsilon >0$
	\begin{equation*}
		\PP\left[\Pi_n > e^{n\rho+n \delta_n} \right] \sim
			\frac{e^{-n\Lambda^*(\rho)}}{ \alpha \sigma(\alpha) \sqrt{2 \pi n}}
			\exp \left\{ - \alpha n \delta_n  - \frac{n \delta_n^2}{2 \sigma^2(\alpha)}
			\left(1 + O(|\delta_n|) \right)   \right\},
	\end{equation*}
	uniformly over $\rho$ and $\{\delta_n\}_{n \geq 0}$ such that
	\begin{equation*}
		\E [\log A] +  \varepsilon  \leq \rho \leq \rho_{\infty} - \varepsilon \quad \mbox{and}
		\quad |\delta_n|\leq \overline{\delta}_n.
	\end{equation*}
	Moreover 	
	\begin{equation*}
		\PP\left[\Pi_{n-j_n} > e^{n\rho+n \delta_n} \right] \sim
			\frac{e^{-n\Lambda^*(\rho)}}{ \alpha \sigma(\alpha) \sqrt{2 \pi n}}
			\exp \left\{ - \alpha n \delta_n  - j_n\Lambda(\alpha)  \right\}, 
	\end{equation*}
	uniformly over $\rho$ and $\{\delta_n\}_{n \geq 0}$ such that
	\begin{equation*}
		\E [\log A] +  \varepsilon  \leq \rho \leq \rho_{\infty} - \varepsilon \quad \mbox{and}
		\quad \max \{n^{1/2}|\delta_n|, n^{-1/2} j_n  \}\leq \overline{\delta}_n.
	\end{equation*}
\end{lem}

The first part of the Lemma \ref{lem: petrov1} is just  the  statement of Petrov's results \cite{petrov1965probabilities}. The second part with a slight  perturbation of time parameter is proved
in \cite{buraczewski2015pointwise} as Lemma 2.4. Note that whenever $\rho_{\infty} = \infty$, then the convergence in Lemma~\ref{lem: petrov1} is uniform in $\rho$ on compact subsets of $(\E \log A, \rho_\infty)$.


Define
$$
	\Pi_{k,n} = \prod_{j=k}^{n-1}A_j, \mbox{ if } k < n \qquad \mbox{ and } \qquad
\Pi_{k,n} = 1, \mbox{ if } k \ge  n.
$$
To prove our main results we intend to compare $Z_n$ with $\Pi_n$, however for technical reasons
(it turns out that it is difficult to control deviations of $Z_n$) it is more convenient to
compare $\Pi_n$ with $Z_{j_n}\Pi_{j_n,n}$ for $j_n = O(\log n)$. Therefore we need a slightly
different version of Petrov's result given by the next Lemma ($H_n$ will be later on replaced by $Z_n$).


\begin{lem}\label{lem:branchPetrov}
Assume we are given $\a\le \a_\infty$ such that $\rho = \rho(\a) >0$.
	Let $H=\{ H_n \}_{n \geq 0}$ be a sequence of non-negative integer-valued random variables
such that $H_k$ is independent of $\Pi_{k,n}$ for every $k$.
Assume that
	\begin{equation}\label{eq:star}
		\EE[H_n^{\alpha}] \sim c_H \lambda(\alpha)^n
	\end{equation} for some $c_H=c_H(\alpha)>0$
	and that for some $\varepsilon_0>0$, one can find a constant $c(\alpha, \varepsilon_0)$ such that
	\begin{equation*}
		\EE[H_n^{s}] \leq c(\alpha,\varepsilon_0) n^{c(\alpha, \varepsilon_0)} \lambda(s)^n \quad \mbox{for all $n$ and } s \in [\alpha, \alpha+\varepsilon_0].
	\end{equation*}
	Then for $j_n = O(\log n )$ 
	\begin{equation*}
		\PP[H_{j_n}\Pi_{j_n,n} > e^{n\rho + n\delta_n}] \sim  \frac{c_H}{ \alpha \sigma(\alpha) \sqrt{2 \pi }}  \frac{e^{-n\Lambda^*(\rho) }}{\sqrt{n}} e^{-\alpha n \delta_n}
	\end{equation*}
	uniformly with respect to
	\begin{equation*}
		 \E [\log A] + \varepsilon  \leq \rho \leq \rho_{\infty} - \varepsilon \quad \mbox{and}
		\quad n^{1/2}|\delta_n| \leq \overline{\delta}_n.
	\end{equation*}
\end{lem}

\begin{proof}  The  proof consists of three steps. \medskip

\noindent
{\sc Step 1: big values of $H_{j_n}$.}  We will show that the set where $H_{j_n}$ attains large values is negligible, i.e.  for $\beta \in (1/2,1)$ we have
	\begin{equation}\label{eq:mm1}
		\PP\left[H_{j_n}\Pi_{j_n,n}>e^{n\rho + n\delta_n}, \: H_{j_n} > e^{\rho j_n + j_n^{\beta}}\right]
			= o \left( \frac{e^{-n\Lambda^*(\rho) -\alpha n\delta_n }}{\sqrt{n}}\right).
	\end{equation}
Dividing values of $H_{j_n}$ into lacunary intervals and invoking independence of $H_{j_n}$ and $\Pi_{j_n,n}$ we obtain 
	\begin{align*}
		\PP\Big[H_{j_n}\Pi_{j_n,n} > & e^{n\rho + n\delta_n}, \: H_{j_n} > e^{ \rho j_n + j_n^{\beta}}\Big]
			 \\ &\leq \sum_{m\geq 0} \PP\left[H_{j_n}\Pi_{j_n,n} >e^{n\rho + n\delta_n}, \:e^{ \rho j_n + j_n^{\beta}} e^m< H_{j_n} \le
				e^{ \rho j_n + j_n^{\beta}}e^{m+1}\right]\\
			 &\leq \sum_{m\geq 0} \PP\left[\Pi_{j_n,n} >e^{n\rho + n\delta_n-\rho j_n -j_n^{\beta}} e^{-m-1} \right]
			\PP\left[H_{j_n} >  e^{ \rho j_n + j_n^{\beta}}e^{m}\right].
	\end{align*}
 Recall  the formula for the Fenchel-Legendre transform:  \begin{equation}\label{eq:fenchel}
\Lambda^*(\rho) = \alpha \rho - \Lambda(\alpha).
\end{equation} Then
	the 
first factor in the term of the series can be bounded using the Markov inequality in the following fashion 
	\begin{equation*}
		\PP\big[\Pi_{j_n,n} >e^{n\rho + n\delta_n} e^{-\rho j_n -j_n^{\beta}} e^{-m-1} \big] \leq
			c e^{-n\Lambda^*(\rho)} e^{-\Lambda(\alpha)j_n-n\alpha\delta_n +\alpha\rho j_n + \alpha j_n^{\beta} + \alpha m  }.
	\end{equation*}
	In order to treat the  second factor, take any $\varepsilon\in (0, \varepsilon_0)$ and write the Taylor expansion
	$\Lambda(\alpha+\varepsilon) = \Lambda(\alpha)+\rho\varepsilon +\frac{\varepsilon^2}{2}\Lambda''(s)$,
	for some $s \in [\alpha, \alpha+\varepsilon]$. Taking $c \geq \sup_{\alpha \in [\alpha, \alpha + \eps_0]} \Lambda''(\alpha) +c(\alpha,\varepsilon_0)$
	we may write
\begin{equation}\label{eq:sr10}
\begin{split}
		\PP\big[H_{j_n} > e^{\rho j_n + j_n^{\beta}}e^{m}\big]
			& \leq \EE[H_{j_n}^{\alpha+\varepsilon}]
			e^{-\rho (\alpha +\varepsilon) j_n -(\alpha + \varepsilon) j_n^{\beta}}e^{-(\alpha+\varepsilon)m} \\
			&\leq c j_n^{c}e^{j_n \Lambda(\alpha+\varepsilon)}
				e^{-\rho (\alpha +\varepsilon) j_n -(\alpha + \varepsilon) j_n^{\beta}}e^{-(\alpha+\varepsilon)m} \\
			&\leq c j_n^c e^{j_n \Lambda(\alpha)  + \varepsilon \rho j_n +c j_n\varepsilon^2}
				e^{-\rho (\alpha +\varepsilon) j_n -(\alpha + \varepsilon) j_n^{\beta}}e^{-(\alpha+\varepsilon)m}\\
			& =cj_n^c\lambda(\alpha)^{j_n} e^{c j_n\varepsilon^2}
				e^{-\rho \alpha j_n -(\alpha + \varepsilon) j_n^{\beta}}e^{-(\alpha+\varepsilon)m}.
	\end{split}
\end{equation}
	If we put this two bounds together and
	sum over $m \geq 0$, we are allowed to infer that
$$		\PP\left[H_{j_n}\Pi_{j_n,n} >e^{n\rho + n\delta_n}, \: H_{j_n} > e^{ \rho j_n + j_n^{\beta}}\right]
\leq  c  \frac{e^{-n\Lambda^*(\rho)}}{1-e^{-\varepsilon}}j_n^c
				\exp\left\{-n\alpha\delta_n-\varepsilon j_n^{\beta} +c j_n\varepsilon^2 \right\}.
$$
	From this point,  the desired bound \eqref{eq:mm1} follows, if one takes  $\varepsilon = (\log n)^{-1/2}$.

\medskip

\noindent
{\sc Step 2: truncated moments of $H_{j_n}$.} We have
	\begin{equation*}
		\lim_{n \to \infty}\lambda(\alpha)^{- j_n}
		\EE \big[H_{j_n}^{\alpha} {\bf 1}_{\{  H_{j_n} \leq e^{ \rho j_n + j_n^{\beta}}\}} \big] = c_H,
	\end{equation*}
	where $c_H$ is the value of the limit in~\eqref{eq:star}. To make this evident, note that evoking \eqref{eq:sr10}
once again, we have
	\begin{align*}
		\EE\big[H_{j_n}^{\alpha} ; \:  H_{j_n} > e^{\rho j_n + j_n^{\beta}}  \big]
		& \leq \sum_{m \geq 0} \EE \big[H_{j_n}^{\alpha}; \:
			e^{ \rho j_n + j_n^{\beta}} e^m< H_{j_n} \leq  e^{ \rho  j_n +j_n^{\beta}}e^{m+1}\big] \\
		& \leq \sum_{m \geq 0} e^{ \alpha\rho j_n + \alpha j_n^{\beta}}e^{\alpha(m+1)}
			\PP\big[ H_{j_n}>e^{ \rho j_n + j_n^{\beta}} e^m\big]\\
		& \leq \frac{1}{1-e^{-\varepsilon}}  c j_n^c \exp\{-\varepsilon j_n^{\beta} +c j_n\varepsilon^2\}\lambda(\alpha)^{j_n}.
	\end{align*}
	Whence, if we put  $\varepsilon = (\log n)^{-\eta}$ for $\eta$ such that $\beta -\eta > 1-2 \eta$ ( $\eta > 1-\beta$) we get
	\begin{equation*}
		\lim_{n \to \infty}\lambda(\alpha)^{-j_n}
			\EE\big[H_{j_n}^{\alpha}{\bf 1}_{\{  H_{j_n} > e^{\rho j_n + j_n^{\beta}}\}}  \big] = 0.
	\end{equation*}
	So the claim in this step follows.

\medskip

\noindent
{\sc Step 3: conclusion.}  In view of {\sc Step 1} it is sufficient to justify
	\begin{equation*}
		\PP\big[H_{j_n}\Pi_{j_n,n} >e^{n\rho + n\delta_n}, \:  H_{j_n} \leq  e^{ \rho j_n + j_n^{\beta}}\big]
			\sim   \frac{c_H e^{-n\Lambda^*(\rho)}}{ \alpha \sigma(\alpha) \sqrt{2 \pi n}} e^{-n\alpha\delta}, \qquad
\mbox{ as } n\to\infty.
	\end{equation*}
For this purpose, 	 by Lemma~\ref{lem: petrov1} and since $H$ is integer valued, we can write
	\begin{align*}
		\PP\Big[H_{j_n}\Pi_{j_n,n} >e^{n\rho + n\delta_n}, \: &  H_{j_n} \leq   e^{ \rho j_n + j_n^{\beta}}\Big]
			 = \PP\left[H_{j_n}\Pi_{j_n,n} >e^{n\rho + n\delta_n}, \: 1\le H_{j_n} \leq  e^{ \rho j_n + j_n^{\beta}}\right]\\
			& = \sum_{k=1}^{\exp\{ \rho j_n + j_n^{\beta}\} }
				\PP\left[\Pi_{j_n,n} >e^{n\rho + n \delta_n-\log k} \right] \PP[H_{j_n}  = k ]\\
			& =\frac{1}{ \alpha \sigma(\alpha) \sqrt{2 \pi n}} e^{-n\Lambda^*(\rho) - n\alpha\delta_n}
				\frac{1+o(1)}{\lambda(\alpha)^{j_n}}
				\EE\big[H_{j_n}^{\alpha} {\bf_ 1}_{\{  H_{j_n} \leq e^{ \rho j_n + j_n^{\beta}}\}}\big].
	\end{align*}
 For the  last inequality we use that the sequence $\overline \delta_n = n^{1/2}\sup_{1\le k \le e^{\rho j_n + j_n^\beta}}|\delta_n - \log k/\log n|$ converges to 0 and refer to the second part of  Lemma \ref{lem:branchPetrov}.
	An appeal to the {\sc Step 2} concludes the proof.
\end{proof}

\subsection{Large deviations of $Z_n$}

Now we are able to present a relatively  short proof of our first result.

\begin{proof}[Proof of Theorem \ref{thm:mthm1}]
To prove the result we estimate deviations of $Z_n$ from its quenched mean $ \Pi_n = \EE[|Z_n | \mathcal {Q}]$. 
 Take 
$n' = K\lfloor \log n\rfloor$ for some large $K$ that will be specified below.
Note that
	\begin{equation}\label{eq:12}
		Z_{n'}\Pi_{n',n} = Z_n + \sum_{k=n'+1}^n (A_{k-1}Z_{k-1} - Z_k)\Pi_{k,n}
	\end{equation}
	and whence by Lemma \ref{lem:s4} and the Chebyshev inequality  for any positive $t$  we have
	\begin{align*}
		\PP\big[\big| Z_n -  Z_{n'}\Pi_{n',n} \big| > t \big]
			& \leq \PP\bigg[\bigg|\sum_{k=n'+1}^n (A_{k-1}Z_{k-1} - Z_k)\Pi_{k,n}\bigg| > \sum_{k=n'+1}^n \frac{t}{2k^2} \bigg]\\
			& \leq \sum_{k=n'+1}^n \PP\bigg[ \big|(A_{k-1}Z_{k-1} - Z_k)\Pi_{k,n} \big|> \frac{ t}{2k^2}  \bigg] \\
			& \leq c \sum_{k=n'+1}^n  \gamma^{k} \lambda(\alpha)^{n-k} k^{2\alpha}  t^{-\alpha}\\
&  \le c \lambda(\alpha)^n t^{-\alpha} \cdot \sum_{k=n'+1}^n \Big( \frac{\gamma}{\lambda(\alpha)} \Big)^k k^{2\alpha}
  \le c\eta^{n'} \lambda(\alpha)^n  t^{-\alpha}
	\end{align*}
for some $\eta = \frac{\gamma}{\lambda(\alpha)} <1$.
 Choosing $t = \delta e^{n\rho}$ for some fixed $\delta>0$ we obtain
	\begin{equation} \label{eq:34}
		\PP\big[\big| Z_n -  Z_{n'}\Pi_{n',n} \big| > \delta e^{n\rho} \big]  \leq
c\delta^{-\alpha} \eta^{n'} e^{-n \Lambda^*(\rho)} = \delta^{-\alpha} \cdot  o\big( {e^{-n\Lambda^*(\rho)}}/ {\sqrt{n}} \big)
	\end{equation}
 if only $K$ is large enough.
 Fix a small $\eps>0$. Then
$$
\P[Z_n>e^{n\rho}]
\le \P\big[ Z_{n'}\Pi_{n',n} > e^{n\rho-\eps} \big] +
		\PP\big[\big| Z_n -  Z_{n'}\Pi_{n',n} \big| > e^{n\rho} (1- e^{-\eps}) \big].
$$
 Combining \eqref{eq:34} with Lemma \ref{lem:branchPetrov} (its hypotheses are satisfied for $H_n = Z_n$ thanks to Lemmas  \ref{lem:s1} and \ref{lem:s4}) yields
$$
\limsup_{n\to\infty} \sqrt n e^{n\Lambda^*(\rho)} \P\big[ Z_n > e^{n\rho}\big] \le \frac{c_He^{\a \eps}}{\alpha \sigma(\alpha) \sqrt{2\pi}}.
$$
Finally  passing with $\eps\to 0$
we obtain
$$
\limsup_{n\to\infty} \sqrt n e^{n\Lambda^*(\rho)} \P\big[ Z_n > e^{n\rho}\big] \le \frac{c_H}{\alpha \sigma(\alpha) \sqrt{2\pi}}.
$$  Similar arguments can be used to estimate the lower limit.
Starting with the inequality
$$
\P[Z_n>e^{n\rho}]
\ge \P\big[ Z_{n'}\Pi_{n',n} > e^{n\rho+\eps} \big] -
		\PP\big[\big| Z_n -  Z_{n'}\Pi_{n',n} \big| > e^{n\rho} ( e^{\eps}-1) \big]
$$ for any $\eps >0$, invoking \eqref{eq:34}, Lemma \ref{lem:branchPetrov} and sending $\eps\to 0$ we arrive at

$$
\liminf_{n\to\infty} \sqrt n e^{n\Lambda^*(\rho)} \P\big[ Z_n > e^{n\rho}\big] \ge \frac{c_H}{\alpha \sigma(\alpha) \sqrt{2\pi}}.
$$  Thus we conclude the result.
\end{proof}

\section{Large deviations of  $T_t^Z$}\label{sec:bp}

\subsection{Preliminary bounds for $T_t^Z$}

To estimate large deviations of $T_t^Z$, similarly as in the previous Section, we
will show the corresponding precise large deviations for the first passage time of
\begin{equation*}
 M_n =  Z_{n'}	\max_{n' \leq j < n} \Pi_{n',j},
\end{equation*}
where $n' = K \lfloor \log n \rfloor $ for big enough $K$ and, as always, $n = \lfloor \rho^{-1} \log t \rfloor$.
We will approximate  $\max_{i\le n}Z_i$ via the mentioned process to conclude the large deviations results for the former which is of our interest.

\begin{lem}\label{lem:approx}
	Suppose that the assumptions of Theorem~\ref{thm:mthm1} are in force and that $n$ and $t$ are related by \eqref{eq:nt}. Then, for any fixed $\delta \in (0,1)$
	$$
		\PP\big[M_{n}>(1+\delta)at, \: \max_{0\leq j \leq n}Z_j \leq at \big] = o\left( \frac{a^{-\a}t^{-
\Lambda^*(\rho)/\rho
}}{\sqrt{\log t}} \right).
	$$
	and
	$$
		\PP\big[M_{n}\leq (1-\delta)at, \: \max_{0\leq j \leq n}Z_j > at \big] = o\left( \frac{a^{-\a}t^{-
\Lambda^*(\rho)/\rho
}}{\sqrt{\log t}} \right)
	$$
	uniformly in $a \in [n^{-K},n^K]$ for any fixed $K>0$.
\end{lem}
\begin{proof} Since the arguments for both claims are similar, we prove here only the first part.
	Consider the following bound
	\begin{equation*}
		\PP\big[
M_{n}>(1+\delta)at, \: \max_{0\leq j \leq n}Z_j \leq at
\big]
\leq
		\sum_{j=n'}^{n}\PP\left[Z_j\leq at, \: Z_{n'}\Pi_{n',j} > (1+\delta)at  \right].
	\end{equation*}
 Combining  \eqref{eq:12}, the Markov inequality and Lemma~\ref{lem:s4} (with $\gamma < \lambda(\a)$) we obtain
\begin{equation}\label{eq:czw3}	
\begin{split}
		\PP\left[Z_j\leq at, \: Z_{n'}\Pi_{n',j} > (1+\delta)at  \right]
			& \leq \PP\left[\sum_{k=n'+1}^j (A_{k-1}Z_{k-1} - Z_k)\Pi_{k,j} > \delta at  \right]\\
			& \leq \sum_{k=n'+1}^j \PP\left[ (A_{k-1}Z_{k-1} - Z_k)\Pi_{k,j} > \frac{\delta at}{2k^2}  \right] \\
			& \leq \sum_{k=n'+1}^jc \gamma^k \lambda(\alpha)^{j-k} k^{2\alpha} \delta^{-\alpha} a^{-\a}t^{-\alpha}\\
	& \leq		c \varepsilon^{n'}\lambda(\alpha)^{j}  \delta^{-\alpha} a^{-\alpha}t^{-\alpha}.
	\end{split}
\end{equation} for some $\eps < 1$.

	If we sum over $n'\leq j \leq n-1$ and recall \eqref{eq:fenchel} we arrive at
	\begin{align*}
		\PP\big[
M_{n}>(1+\delta)at, \: \max_{0\leq j \leq n}Z_j \leq at
\big]
 			&\le  \sum_{j=n'}^{n}\PP\left[Z_j\leq at, \: Z_{n'}\Pi_{n',j} > (1+\delta)at  \right] \\
			& \leq \sum_{j=n'}^{n}c \varepsilon^{n'}\lambda(\alpha)^{j}  \delta^{-\alpha} a^{-\a}t^{-\alpha}\\
			&\leq  c \varepsilon^{n'}  \delta^{-\alpha} a^{-\a}t^{-
\Lambda^*(\rho)/\rho
} = o\left( \frac{a^{-\a}t^{-
\Lambda^*(\rho)/\rho
}}{\sqrt{\log t}} \right)
	\end{align*}
	if only $K$ is big enough.
\end{proof}

For $j\ge n'$ define $X_j=Z_{n'} \Pi_{n',j}$.

\begin{lem}\label{lem:burmas1}
	Let $L$ and $N$ be two integers such that  $L \geq 1$ and  $-1 \leq N \leq L$ and let $K$ be a fixed constant.  Then for
	$\beta < \alpha$ such that $\lambda(\beta) < \lambda(\alpha)$ and sufficiently large $t$,
	\begin{equation*}
 \sup_{n^{-K} \le a,b\le n^K} b^\b a^{\a-\b}		
 \PP\left[M_{n - L} > at,\: X_{n-N} > b t \right]
		\leq c(\alpha,\beta) \lambda(\alpha)^{-L} \lambda(\beta)^{L - N}
		\frac{t^{-\Lambda^*(\rho)/\rho}}{\sqrt{\log t}}.
	\end{equation*}
\end{lem}

\begin{proof}
	We may write
	\begin{equation*}
		\PP\left[M_{n - L} > at, \: X_{n-N} > b t \right]  \leq
			\sum_{j=n'-1}^{n -L} \PP\left[ X_j > a t, \: X_{n-N} > bt \right].
	\end{equation*}
	We will bound the right-hand side of above inequality term by term. To do so, we will need to distinguish between big and
	small values of $j$. Put $\delta = \frac{\lambda(\beta)}{\lambda(\alpha)} < 1$. We will compare $n-j-L$ with
	$Q \log n$ for some
	integer $Q$ such that $\delta^{- Q \log n} > \log t$. \medskip

\noindent
{\sc Step 1.} First we present a bound for  $j \leq \lfloor n - Q \log n-L \rfloor = n_1$. Since $X_{n-N} = X_j \Pi_{j,n-N}$ for $j \leq n-L$, we can write
	\begin{align*}
		\PP\big[ X_{j} > a t, \: X_{n-N} > bt \big]
 		& =  \sum_{m=0}^{\infty}
			\PP\big[e^{m+1}at \geq X_j > e^mat, \: X_{n-N} > bt\big ]\\
 		 &\leq \sum_{m=0}^{\infty}
			\PP\big[ X_j > e^mat, \: \Pi_{j,n-N} > {b}/{a} e^{-m-1} \big] \\
 		& \leq  \sum_{m=0}^{\infty} \PP\left[X_j > e^mat \right]\:
			\PP \left[ \Pi_{j,n-N} > {b}/{a} e^{-m-1}\right] \\
  		& \leq \sum_{m=0}^{\infty}c \lambda(\alpha)^j e^{-\alpha m} a^{-\alpha}t^{-\alpha}
			\lambda(\beta)^{n-N-j}b^{-\beta} a^{\beta}e^{\beta(m+1)}\\
  		& \leq c(\alpha, \beta)  b^{-\beta} a^{\beta -\alpha} \delta^{n-j-L} t^{-\Lambda^*(\rho)/\rho}
			\lambda(\alpha)^{-L} \lambda(\beta)^{L-N}.
	\end{align*}
	If we now sum over $j\leq \lfloor n - Q \log n-L \rfloor = n_1$, we will arrive at
	\begin{multline*}
		 \sum_{0 \leq  j \leq n_1} \PP\left[ X_{j} > a t, \: X_{n-N} > bt \right]
		 \leq c \sum_{0 \leq  j \leq n_1}b^{-\beta} a^{\beta -\alpha} \delta^{n-j-L} t^{-\Lambda^*(\rho)/\rho}
			\lambda(\alpha)^{-L} \lambda(\beta)^{L-N} \\
			\leq c b^{-\beta} a^{\beta -\alpha} \delta^{Q \log n} t^{-\Lambda^*(\rho)/\rho}
			\lambda(\alpha)^{-L} \lambda(\beta)^{L-N}
  		\leq  c b^{-\beta} a^{\beta -\alpha}  \frac{t^{-\Lambda^*(\rho)/\rho}}{\log t}
			\lambda(\alpha)^{-L} \lambda(\beta)^{L-N}.
	\end{multline*}

\noindent
{\sc Step 2.} Now we give a bound for $ j > \lfloor n - Q \log n-L \rfloor = n_1$. Take $B>0$ to be any constant such that  $-\alpha B +1 < 0$, whenever
	$\alpha \geq \alpha_{0}$ and $-\alpha B +1 - \Lambda(\alpha)Q < 0$ for $\alpha < \alpha_{0}$. Consider the
	decomposition
	\begin{align*}
		\PP\left[ X_j > at, \: X_{n-N} > bt \right] & \leq
			\PP\left[ X_j > a te^{B \log n}, \: X_{n-N}>bt \right] \\
			& + \PP\left[ at e^{B \log n}\geq X_j >at, \: X_{n-N} > bt \right] = I_1 + I_2.
	\end{align*}
	Invoke  Markov's inequality in order to bound the first term viz.
	\begin{align*}
		I_1&   \leq \PP\left[ X_j \geq at e^{B \log n} \right]
			\leq  c t^{-\alpha} a^{-\alpha}n^{-\alpha B} \lambda(\alpha)^{j}
			=  ct^{-\Lambda^*(\rho)/\rho}a^{-\alpha} n^{-\alpha B} \lambda(\alpha)^{j-n}\\
		&  \leq c t^{-\Lambda^*(\rho)/\rho} a^{-\alpha} \frac{1}{\log t} n^{-\alpha B + 1}  \lambda(\alpha)^{-n+j+L} \lambda(\alpha)^{-L}
 			\leq ca^{-\alpha} t^{-\Lambda^*(\rho)/\rho} \frac{1}{ \log t} \lambda(\alpha)^{-L}.
	\end{align*}
	Turning our attention to the second term $I_2$ we apply Lemma \ref{lem: petrov1}, which gives the uniform estimates,
combined with the same
	procedure as the one used in the previous step.
	\begin{align*}
		I_2  =\PP\big[ at e^{B \log n}\geq X_j >at, & \:X_{n-N} > bt \big]
		 \leq \sum_{m=0}^{\left \lceil{B \log  n - 1}\right \rceil }
			\PP\left[at e^{m+1} \geq X_j > at e^m, \:X_{n-N} > bt\right]\\
		& \leq \sum_{m=0}^{\left \lceil{B \log  n - 1}\right \rceil } \PP\left[X_j > ate^m \right]\:
			\PP \left[ \Pi_{j, n - N} > \frac{b}{a}e^{-m-1}\right]\\
		& \leq \sum_{m=0}^{\left \lceil{B \log  n - 1}\right \rceil }
			 c \frac{t^{-\Lambda^*(\rho)/\rho}}{\sqrt{\log t}} \lambda(\alpha)^{j-n} a^{-\alpha}e^{-\alpha m}
			b^{-\beta}a^{\beta} e^{\beta (m+1)} \lambda(\beta)^{n-j-N} \\
		& \leq \sum_{m=0}^{\left \lceil{B \log  n - 1}\right \rceil }
			c b^{-\beta}a^{\beta-\alpha}\frac{t^{-\Lambda^*(\rho)/\rho}}{\sqrt{\log t}}
 			e^{(\beta-\alpha) m} \delta^{n-j-L}   \lambda(\alpha)^{-L} \lambda(\beta)^{L-N}  \\
		& \leq  c b^{-\beta}a^{\beta-\alpha} \frac{t^{-\Lambda^*(\rho)/\rho}}{\sqrt{\log t}}  \delta^{n-j - L}
			\lambda(\alpha)^{-L} \lambda(\beta)^{L-M}.
	\end{align*}
	Now combine bound for $I_1$ and $I_2$ to get
	\begin{equation*}
  		 \PP\left[ X_j > at, \: X_{n-N} > bt \right]
 		 \leq  c a^{-\alpha} t^{-\Lambda^*(\rho)/\rho} \frac{1}{ \log t} \lambda(\alpha)^{-L}
		+   c b^{-\beta}a^{\beta-\alpha} \frac{t^{-\Lambda^*(\rho)/\rho}}{\sqrt{\log t}}  \delta^{n-j-L}
			\lambda(\alpha)^{-L} \lambda(\beta)^{L-N}.
	\end{equation*}
	Summing over $n-L \geq j > n- Q \log n-L$ establishes a bound sufficient for our needs
	\begin{multline*}
		\sum_{ n-L \geq j \geq n_1 }  \PP\left[ X_j > at, \: X_{n-N} > bt \right]\\
		\leq  ca^{-\alpha} t^{-\Lambda^*(\rho)/\rho} \frac{\log n}{ \log t} \lambda(\alpha)^{-L}
		 +   cb^{-\beta}a^{\beta-\alpha} \frac{t^{-\Lambda^*(\rho)/\rho}}{\sqrt{\log t}}
			\lambda(\alpha)^{-L} \lambda(\beta)^{L-N}.
	\end{multline*}
\noindent
{\sc Step 3.} Lastly, combine the claims of previous steps to derive that for sufficiently large $t$,
	\begin{align*}
		\PP\left[M_{n - L} > at,\: X_{n-N} > b t \right]  & \leq
			c b^{-\beta}a^{\beta-\alpha}  \frac{t^{-\Lambda^*(\rho)/\rho}}{\log t} \lambda(\alpha)^{-L} \lambda(\beta)^{L-N} \\
			& +c a^{-\alpha} t^{-\Lambda^*(\rho)/\rho} \frac{\log n}{ \log t} \lambda(\alpha)^{-L}
		+   c b^{-\beta}a^{\beta-\alpha} \frac{t^{-\Lambda^*(\rho)/\rho}}{\sqrt{\log t}}
			\lambda(\alpha)^{-L} \lambda(\beta)^{L-N}\\
		& \leq  c  b^{-\beta}a^{\beta-\alpha} \lambda(\alpha)^{-L} \lambda(\beta)^{L - N}   \frac{t^{-\Lambda^*(\rho)/\rho}}{\sqrt{\log t}}.
	\end{align*}
This completes the proof.
\end{proof}

\subsection{Lower and upper estimates}
We will focus our attention on establishing that first passage time for $\{X_j \}_{j \geq n'}$ is of the correct order.

\begin{prop}\label{prop:weakasymp}
For any constant $K$ there is a positive constant $c$ such that for
$t$ big enough one has
	\begin{equation}\label{eq:weakasymp}
		\frac{1}{c}\frac{t^{-\Lambda^*(\rho)/\rho}}{\sqrt{\log t}} \leq
		a^\a \PP\left[M_{n-1} \leq at, \: X_{n} > at \right] \leq
		c \frac{t^{-\Lambda^*(\rho)/\rho}}{\sqrt{\log t}}
	\end{equation}
uniformly for $a\in [n^{-K}, n^K]$.
\end{prop}

\begin{proof}
	Firstly, note that the upper bound in~\eqref{eq:weakasymp} follows by

invoking Lemma \ref{lem:burmas1} 
	\begin{equation*}
		\PP\left[M_{n-1} \leq at, \: X_n > at \right] \leq  c \frac{a^{-\alpha}t^{-\Lambda^*(\rho)/\rho}}{\sqrt{\log t}}.
	\end{equation*}

	To establish the lower bound, we denote for non-negative integer $L$ and any pair of positive reals $ 0  < \gamma < r <1$
	\begin{equation*}
		\mathcal{A}=\mathcal{A}(r,\gamma, L) = \left\{  \max_{n-L \leq j \leq  n-1}\Pi_{n-L,j} \leq r\gamma^{-1}, \: \Pi_{n-L,n} > \gamma^{-1}  \right\}
	\end{equation*}
	and
	\begin{equation*}
		\mathcal{B}=\mathcal{B}(r,\gamma, L) = \left\{  M_{n-L-1} \leq at, \: \gamma a t \leq X_{n-L-1} < \gamma r^{-1} at  \right\}.
	\end{equation*}
	Then, by a direct calculations, it can be easily verified that
	\begin{equation*}
		\mathcal{A} \cap \mathcal{B} \subseteq \left\{ M_{n-1} \leq at, \: X_n > at  \right\}.
	\end{equation*}
	By independence of $\mathcal{A}$ and $\mathcal{B}$ we are allowed to treat probabilities of respective events separately.
	To bound $\PP[\mathcal{B}]$ note that
	\begin{equation*}
		\PP[\mathcal{B}]  = \PP[\gamma at \leq X_{n-L-1} < \gamma r^{-1} at ]
			- \PP\left[M_{n-L-1} > at, \: \gamma at \leq X_{n-L+1} < \gamma r^{-1} at \right].
	\end{equation*}
	The first probability, by Lemma~\ref{lem:branchPetrov} exhibits the following asymptotic behaviour
	\begin{equation*}
		\PP[\gamma at \leq X_{n-L-1} < \gamma r^{-1} at ] \sim c_{1}(\alpha, r) \gamma^{-\alpha}\lambda(\alpha)^{-L}
			\frac{a^{-\alpha}t^{-\Lambda^*(\rho)/\rho}}{\sqrt{\log t}}
	\end{equation*}
	while asymptotic of the second probability can be bounded by another appeal to Lemma~\ref{lem:burmas1} with $N=L$
	\begin{align*}
		\PP\left[M_{n-L-1} > at, \: \gamma a t \leq X_{n-L} < \gamma r^{-1} at \right] &\leq
		\PP\left[M_{n-L-1} > at, \:  X_{n-L} > \gamma  at \right]\\ &\le c_{2}(\alpha, \beta) \gamma^{-\beta}\lambda(\alpha)^{-L}
			\frac{a^{-\alpha}t^{-\Lambda^*(\rho)/\rho}}{\sqrt{\log t}}.
	\end{align*}
	If we put everything together, we will arrive at the conclusion that, uniformly in $a \in [n^{-K}, n^K]$
	\begin{equation*}
		\PP \left[ M_{n-1} \leq a t, \: X_n > a t  \right] \geq \PP[\mathcal{A}]a^{-\alpha} \left( c_{1}(\alpha, r) \gamma^{-\alpha} - c_{2}(\alpha,\beta)\gamma^{-\beta}+o(1)\right)\lambda(\alpha)^{-L}
			\frac{t^{-\Lambda^*(\rho)/\rho}}{\sqrt{\log t}}.
	\end{equation*}
	We need to ensure, that for a proper choice of $\gamma, r$ and $L$,
	\begin{equation*}
		\PP[\mathcal{A}] \left( c_{1}(\alpha, r) \gamma^{-\alpha} - c_{2}(\alpha,\beta)\gamma^{-\beta}\right) > 0.
	\end{equation*}
	To do so, first take $r$ such that $\PP[A>r^{-2}]>0$ and then take $\gamma$ sufficiently small such that
	\begin{equation*}
		c_{1}(\alpha, r) \gamma^{-\alpha} - c_{2}(\alpha,\beta)\gamma^{-\beta}>0\quad \mbox{and} \quad \gamma < r^2
	\end{equation*}
	Finally choose $L$ such that 
(changing $\gamma$ if necessary)
	$\PP[ r^2 \gamma^{-1} < A^L <  r \gamma^{-1}]>0$. The constants chosen in this way allow us to write
	\begin{align*}
\PP[{\mathcal A}] &= \PP \Big[ \max_{1\le j \le L-1}   \Pi_j  \le r\gamma^{-1}, \Pi_L > \gamma^{-1}  \Big]\\
&\ge \PP\Big[ A_L > r^{-2}, \Pi_{L} = \max_{1\le j \le L-1} \Pi_j, r^2\gamma^{-1} < \Pi_{L} < r \gamma^{-1} \Big]\\
&\ge \PP\Big[ A_L > r^{-2}, \Pi_{L} \ge \Pi_{L-2}\ge \ldots \ge \Pi_1, r^2\gamma^{-1} < \Pi_{L} < r \gamma^{-1} \Big]\\
&\ge \PP\big[ A_L > r^{-2}\big] \prod_{i=1}^{L-1} \PP \big[ r^2\gamma^{-1} < A_{i}^L < r \gamma^{-1} \big]\\
&>0.
	\end{align*}
\end{proof}

\subsection{Conclusions}
\begin{lem}\label{lem:concl}
	We  have
	\begin{equation*}
		 \lim_{t \to \infty} a^{\alpha} t^{\Lambda^*(\rho)/\rho} \sqrt{\log t}\PP \left[ M_{n-1 } \leq a t, \: X_{n} > a t  \right] = \cC_4(\rho)
	\end{equation*}
	uniformly for $a\in [n^{-K}, n^K]$, where $K$ is a fixed constant.
\end{lem}

\begin{proof}
Take  large $L$ and write
	\begin{multline*}
		\PP\left[ Z_{n'}\Pi_{n',n -L} \max_{n-L \leq j < n}\Pi_{n-L,j} \leq at,\: Z_{n'}\Pi_{n',n }  >at \right]   = \\
   			= \PP\left[ Z_{n'}\Pi_{n',n -L } \max_{n-L\leq j<n}\Pi_{n-L,j} \leq at,\: Z_{n'}\Pi_{n',n }  >at,
				\: \max_{j<n-L}Z_{n'}\Pi_{n',j} > at \right] \\
			+\PP \left[ M_{n-1 } \leq t, \: Z_{n'}\Pi_{n', n} > at  \right].
	\end{multline*}
	Using Lemma~\ref{lem:burmas1} we infer that the first term on the right-hand side has arbitrarily small contribution since it
	can be bounded uniformly with respect to $a$ viz.
	\begin{multline*}
		\PP\left[Z_{n'}\Pi_{n',n }  >at,\: \max_{j<n-L}Z_{n'}\Pi_{n',j} > at \right] = \PP\big[ M_{n-L}>at, X_n > at \big]\\
		\leq c(\alpha, \beta) \lambda(\alpha)^{-L} \lambda(\beta)^{L}  \frac{a^{-\alpha}t^{-\Lambda^*(\rho)/\rho}}{\sqrt{\log t}}
		 = c(\alpha, \beta) \delta^{L}  \frac{a^{-\alpha}t^{-\Lambda^*(\rho)/\rho}}{\sqrt{\log t}},
	\end{multline*}
 where $\beta<\a$ and 
$\delta = \frac{\lambda(\beta)}{\lambda(\alpha)} < 1$. Choosing large $L$, $\delta^L$ can be 	
	arbitrary small. Whence to conclude the Lemma 
it will be sufficient to show that
	\begin{equation*}
	a^{\alpha} 	t^{\Lambda^*(\rho)/\rho}\sqrt{\log t}\;\PP\left[ Z_{n'}\Pi_{n',n -L-1}
			\max_{n-L \leq j < n}\Pi_{n-L,j} \leq at,\: Z_{n'}\Pi_{n',n }  >at \right]
	\end{equation*}
	converges for $L$ arbitrarily large. For this reason note that the probability in question can be decomposed in the following fashion
	\begin{align*}
		\PP\Big[ Z_{n'}&\Pi_{n',n -L}  \max_{n-L \leq j < n}\Pi_{n-L,j} \leq at,\: Z_{n'}\Pi_{n',n }  >at \Big] \\
  		=& \PP\Big[Z_{n'}\Pi_{n',n -L} \leq ate^{-(\log t)^{1/4}},\: Z_{n'}\Pi_{n',n -L} \max_{n-L \leq j < n}\Pi_{n-L,j} \leq at,
			\: Z_{n'}\Pi_{n',n }  >at \Big]\\
		 +&\PP\Big[ ate^{-(\log t)^{1/4}}<Z_{n'}\Pi_{n',n -L}\le at,\:Z_{n'}\Pi_{n',n -L} \max_{n-L \leq j < n}\Pi_{n-L,j} \leq at,
			\: Z_{n'}\Pi_{n',n }  >at \Big]\\ =& J_1+J_2.
	\end{align*}
 Consider the first term for the moment. Our aim is to prove that  its contribution is negligible. We will utilize the same procedure as
	 the one used in the proof of Lemma~\ref{lem:burmas1}. For $\beta > \alpha$ we have
	\begin{align*}
		J_1	& \leq  \PP\left[Z_{n'}\Pi_{n',n -L} \leq ate^{-(\log t)^{1/4}},\: Z_{n'}\Pi_{n',n }  >at \right] \\
 			&\leq \sum_{m \geq 0} \PP\left[e^{-m-1}ate^{-(\log t)^{1/4}}
				\leq Z_{n'}\Pi_{n',n -L} \leq e^{-m}ate^{-(\log t)^{1/4}},\: Z_{n'}\Pi_{n',n }  >at \right]\\
 			& \leq \sum_{m \geq 0} \PP\left[e^{-m-1}ate^{-(\log t)^{1/4}}
				\leq Z_{n'}\Pi_{n',n -L}\right] \:  \PP\left[ \Pi_{n-L,n }  >e^{m}e^{(\log t)^{1/4}} \right] \\
			& \leq \sum_{m \geq 0}  {c}\lambda(\alpha)^{n - L}a^{-\alpha} t^{-\alpha}  e^{\alpha (\log t)^{\frac{1}{4}}}
				e^{\alpha(m+1)} \lambda(\beta)^{L+1} e^{-\beta (\log t)^{\frac{1}{4}}} e^{-\beta m } \\
 			& =  c a^{-\alpha}t^{-\Lambda^*(\rho)/\rho}  e^{(\alpha - \beta)(\log t)^{\frac{1}{4}}}
				   \bigg(\frac{\lambda(\beta)}{\lambda(\alpha)}\bigg)^{L}  \sum_{m \geq 0} e^{\alpha(m+1)} e^{-\beta m }
				= o\left( \frac{a^{-\alpha} t^{-\Lambda^*(\rho)/\rho}}{\sqrt{\log t}} \right),
	\end{align*}
	for some $c = c(\alpha,\beta, L)$. Left with an~investigation of $J_2$ we note that by the same arguments
	as above one can deduce that
	\begin{equation*}
		\PP \left[ Z_{n'}\Pi_{n',n -L} > ate^{-(\log t)^{1/4}},\: \Pi_{n-L,n } > e^{(\log t)^{1/4}}\right]
		= o\left( \frac{ {a^{-\alpha}}  t^{-\Lambda^*(\rho)/\rho}}{\sqrt{\log t}} \right)
	\end{equation*}
	and as a consequence
	\begin{align*}
		J_2  = &\PP\left[ ate^{-(\log t)^{1/4}} < Z_{n'}\Pi_{n', n-L}\le at ,
			\: \right. Z_{n'}\Pi_{n', n-L} \max_{n-L\leq j<n}\Pi_{n-L,j} \leq at, \\
			 & \left.\: Z_{n'}\Pi_{n', n} >at, \Pi_{n-L,n} \le e^{(\log t)^{1/4}}\right]
		+ o\left( \frac{{a^{-\alpha}}  t^{-\Lambda^*(\rho)/\rho}}{\sqrt{\log t}} \right).
	\end{align*}
	By conditioning on $M_{L-1}'=\max_{n-L+1\leq j<n}\Pi_{n-L+1,j}$ and $\Pi'_L = \Pi_{n-L+1,n}$ we  have
	\begin{equation}
 		J_2=   \int_{0\le y\le x < e^{(\log t)^{1/4}}}
			\PP\left[at x^{-1} < Z_{n'}\Pi_{n',n -L} < aty^{-1} \right]
			\PP \left[M'_{L-1} \in dy, 
 \Pi'_L\in dx  \right]
		+ o\left( \frac{a^{-\a}t^{-\Lambda^*(\rho)/\rho}}{\sqrt{\log t}} \right).
	\end{equation}
Now apply Lemma~\ref{lem:branchPetrov} 
with $j_n = L+K\log n$,
	$\overline{\delta}_n = C n^{-\frac{1}{4}}$ and  $\delta_n = \frac{\log(ay^{-1})}{n}$, where $C>0$ is such that $|\sqrt{n}\delta_n |= n^{-1/2}\log(y/a) \leq n^{-1/4}C$. We infer that
	\begin{equation*}
		\PP\left[Z_{n'}\Pi_{n',n - L} \geq aty^{-1}  \right]  = c(\alpha) \frac{a^{-\alpha}t^{-\Lambda^*(\rho)/\rho}}{\sqrt{\log t}} y^\alpha
		e^{-L \Lambda(\alpha)} (1+o(1)).
	\end{equation*}
	We conclude that
	\begin{equation*}
		J_2= C(\alpha) \frac{a^{-\alpha}t^{-\Lambda^*(\rho)/\rho}}{\sqrt{\log t}} e^{-L \Lambda(\alpha)}
			\EE\left[\left( (\Pi'_L)^{\alpha} - (M_{L-1}')^{\alpha} \right)_{+} \right]   (1+o(1)) \quad \text{as } t \to \infty.
	\end{equation*}
\end{proof}

\begin{proof}[Proof of Theorem~\ref{thm:mthm2}]
We focus on the proof of precise pointwise estimates \eqref{eq:mthm2point} (Step 1), since it implies almost immediately both \eqref{eq:2.4down} and \eqref{eq:2.4up} (Step 2 and Step 3). Let us mention that  both limits  \eqref{eq:2.4down} and \eqref{eq:2.4up} can be proved in a much simpler way, e.g. using similar techniques to those presented in Lalley \cite{Lalley} and in particular omitting the tedious proofs of Lemmas \ref{lem:burmas1} and \ref{lem:concl}.

\medskip

{\sc Step 1.}
  First we prove that for any  fixed constant $K$ we have
  \begin{equation}\label{eq:precise}
	\P\big[ T_{at}^Z = n\big] = 	\PP\left[ \max_{j \leq n-1}Z_j\leq at, \: Z_n > at \right] \sim \cC_4(\rho) \frac{a^{-\alpha}t^{-\Lambda^*(\rho)/\rho}}{\sqrt{\log t}},
	\end{equation} uniformly for $a\in[n^{-K}, n^K]$. Observe that formula \eqref{eq:precise} implies \eqref{eq:mthm2point}. Indeed for $n,t,\Theta$ as in \eqref{eq:nt}
$$
\P\big[ T_t^Z = n \big] = \P\Big[ \max_{j \leq n-1}Z_j\leq e^{\rho \Theta(t)} t, \: Z_n > e^{\rho \Theta(t)}t
\Big]
\sim \cC_4(\rho) \lambda(\a)^{-\Theta(t)} \frac{t^{-\Lambda^*(\rho)/\rho }}{\sqrt{\log t}}.
$$
Note that \eqref{eq:precise} is indeed much stronger than $\eqref{eq:mthm2point}$, because the estimates are uniform, however uniformity will be needed to deduce \eqref{eq:2.4down} and \eqref{eq:2.4up}.

 Recall that $n'=K\log n$ (the constant $K$ will come into play below).
	In view of all previous considerations, we are left with approximation of $Z_n$ with  $X_n = Z_{n'}\Pi_{n',n}$ as $n \to \infty$.

\medskip

\noindent
{\sc Step 1a.}  We  prove upper estimate
\begin{equation}\label{eq:5}
  \limsup_{t\to\infty}  a^{\alpha} t^{\Lambda^*(\rho)/\rho} \sqrt{\log t} \PP\left[ T_{at}^{Z} = n\right] \le \cC_4(\rho).
\end{equation}
  For this purpose we fix $\delta>0$ and   write
\begin{align*}
 \PP\Big[ \max_{j \leq n-1}Z_j\leq at, \: Z_n > at \Big]
&\le 	\PP\big[\max_{j \leq n-1}Z_j\leq at, \: M_{n-1} > (1+\delta)at  \big]
+\PP\big[X_n \le  (1- \delta)at, Z_n>at  \big]\\
&+
 \PP\big[ (1-\delta)at <  X_n  \le   (1+ \delta)at \big]
+ \PP\big[M_{n-1} \le  (1+\delta)at, X_n \ge  (1+ \delta)at \big]\\
&= I + II + III + IV.
\end{align*}
In view of Lemma \ref{lem:concl} the last expression has the required asymptotic behaviour, that is

\begin{equation}\label{eq:czw2}
 \lim_{t \to \infty}  a^\a t^{\Lambda^*(\rho)/\rho}   IV = (1+\delta)^{-\a} \cC_4(\rho).
\end{equation}
Thus we need to prove that the other terms are negligible.
 The first one, namely I,
is of the order
	\begin{equation}\label{eq:1I}
		I=o\left( \frac{ a^{-\alpha}t^{-\Lambda^*(\rho)/\rho}}{\sqrt{\log t}} \right)
	\end{equation}
	  by the merit of Lemma~\ref{lem:approx}.
 Arguing as in the proof of Lemma~\ref{lem:approx} (see \eqref{eq:czw3}), one can deduce
	\begin{equation}\label{eq:1II}
		 II = \PP\left[ Z_n>at, \: X_n
 \leq (1-\delta)at \right]=o\left( \frac{ a^{-\alpha}t^{-\Lambda^*(\rho)/\rho}}{\sqrt{\log t}} \right).
	\end{equation}
	By an appeal to Lemma~\ref{lem:branchPetrov} we estimate III viz.
	\begin{equation}\label{eq:1III}
		\PP[(1-\delta)at < X_n 
\leq (1+\delta)at] \leq
			 h(\delta) \frac{ a^{-\alpha}t^{-\Lambda^*(\rho)/\rho}}{\sqrt{\log t}},
	\end{equation}
	where $h(\delta) \to 0$ as $\delta \to 0$.  Combining \eqref{eq:czw2}, \eqref{eq:1I}, \eqref{eq:1II} with \eqref{eq:1III}
 and  then passing with $\delta$ to 0 we conclude \eqref{eq:5}.

\medskip

\noindent

{\sc Step 1b.}
To get the lower bound, apply the same procedure.  However in this case, for fixed $\delta>0$ we use the following
inequality
\begin{multline*}
 \PP\Big[ \max_{j \leq n-1}Z_j\leq at, \: Z_n > at \Big]
\ge \PP\big[M_{n-1} \le  (1-\delta)at, X_n \ge  (1- \delta)at \big]\\
- 	\PP\big[\max_{j \leq n-1}Z_j >  at, \: M_{n-1} \le (1-\delta)at  \big]
 - \PP\big[ (1-\delta)at <  X_n  \le   (1+ \delta)at \big]
- \PP\big[X_n  >   (1+ \delta)at, Z_n \le at  \big]
\end{multline*}
The same arguments as in the Step 1a  (the first term on the right side above dominates, whereas all the remaining are
negligible) lead us to the lower bound
\begin{equation*}
  \liminf_{t\to\infty} a^{\alpha} t^{\Lambda^*(\rho)/\rho} \sqrt{\log t} \PP\left[ T_{at}^{Z} = n\right] \ge  \cC_4(\rho),
\end{equation*}
which together with \eqref{eq:5} completes proof of the first step.

\medskip

\noindent
{\sc Step 2.} Now we prove \eqref{eq:2.4down}. Note, that we only consider $\rho > \rho_0$, so in particular $\lambda(\alpha)>1$.
 We will use below the following bound
	$$
		\PP[\max_{0\leq j \leq \overline{n}} Z_j > t] \leq \PP[W_{\overline{n}}>t] = o\left( \frac{t^{-\Lambda^*(\rho)/\rho}}{\sqrt{\log t}}\right),$$
for $\overline{n} = n -D \log n$ with	$D > \Lambda(\alpha)^{-1}(2\alpha + 7)$. We postpone its proof to the next Section (see Corollary \ref{cor:nbar}).
Therefore, by \eqref{eq:precise}
\begin{align*}
\PP[T_t^Z \le n] &\sim
\PP[n- D\log n \le T_t^Z \le n] \\
&= \sum_{j=0}^{D\log n}\PP[\max_{i < n-j } Z_i \le t \mbox{ and } Z_{n-j}> t]\\
&=\sum_{j=0}^{D\log n}\PP[\max_{i < n-j } Z_i \le e^{\rho(n-j)}e^{\rho j}e^{\rho\Theta(t)} \mbox{ and }Z_{n-j}> e^{\rho(n-j)}e^{\rho j}e^{\rho\Theta(t)}]\\
&\sim \cC_4(\rho) \sum_{j=0}^{D\log n} \frac{e^{-\a \rho j} e^{-\a\rho\Theta(t)}  e^{-\Lambda^*(\rho) (n-j) }}{\sqrt{\rho(n-j)}}\\
&\sim \cC_4(\rho) \lambda(\a)^{-\Theta(t)}  \frac{ t^{-\Lambda^*(\rho)/\rho} }{\sqrt{\log t}} \sum_{j=0}^{D\log n}  \lambda(\a)^{-j}\\
&\sim \frac{\cC_4(\rho) \lambda(\alpha)^{1-\Theta(t)}}{\lambda(\alpha)-1} \frac{ t^{-\Lambda^*(\rho)/\rho} }{\sqrt{\log t}},
\end{align*}
which completes the proof of \eqref{eq:2.4down}.

\medskip

\noindent
 {\sc Step 3.} To prove \eqref{eq:2.4up} we proceed as above. This time $\rho < \rho_0$ and $\lambda(\alpha)<1$. Applying the  Markov  inequality and reasoning as in Lemma \ref{lem:m1}
one can  prove that for large $D$ and  $\overline{n} = n + D \log n$
	 $$
		\PP[\max_{ j > \overline{n}} Z_j > t]
= o\left( \frac{t^{-\Lambda^*(\rho)/\rho}}{\sqrt{\log t}}\right)
	$$
Next, applying \eqref{eq:precise}
\begin{align*}
\PP[T_t^Z \ge n] &\sim
\PP[n \le T_t^Z \le n +  D\log n]\\
&= \sum_{j=0}^{D\log n}\PP[\max_{i < n+j } Z_i \le t \mbox{ and } Z_{n+j}> t]\\
&=\sum_{j=0}^{D\log n}\PP[\max_{i < n+j } Z_i \le e^{\rho(n+j)}e^{- \rho j} e^{\rho\Theta(t)} \mbox{ and }Z_{n+j}> e^{\rho(n+j)}e^{-\rho j}e^{\rho\Theta(t)}]\\
&\sim \cC_4(\rho) \sum_{j=0}^{D\log n} \frac{e^{\a \rho j}e^{- \a \rho\Theta(t)}  e^{-(n+j)\Lambda^*(\rho) }}{\sqrt{\rho(n+j)}}\\
&\sim \cC_4(\rho) \frac{\lambda(\a)^{-\Theta(t)} t^{-\Lambda^*(\rho)/\rho} }{\sqrt{\log t}} \sum_{j=0}^{D\log n}  \lambda(\a)^{j}\\
&\sim \frac{\cC_4(\rho) \lambda(\a)^{-\Theta(t)} }{1- \lambda(\alpha)} \frac{ t^{-\Lambda^*(\rho)/\rho} }{\sqrt{\log t}}.
\end{align*}
 We conclude the proof.
\end{proof}

\section{Estimates for $W$}\label{sec:total}

\subsection{Preliminaty bounds for the total population $W_n$ and the perpetuity $R_n$}
In this Section we consider the total population up to generation $n$ of BPRE $Z$, $W_n = \sum_{k=0}^nZ_k$. Recall that its quenched expectation $R_n = \E[W_n|{\mathcal Q}]$ forms a perpetuity sequence (see \eqref{eq:wt1}). We start with a few elementary properties of $W_n$ and $R_n$.
The following Lemma was proved in \cite{buraczewski2014large} (see proof of Theorem 2.1).
\begin{lem}\label{lem:momrn}
	If $\alpha >\alpha_0$, then  there exists a finite positive limit
	\begin{equation*}
		c_R(\alpha)  = \lim_{n \to \infty}\lambda(\alpha)^{-n}\EE[R_n^{\alpha}].
	\end{equation*}
\end{lem}
 The next Lemma was stated in \cite{buraczewski2014large} as Lemma~3.1. We prove below  its counterpart in terms of $W$.
\begin{lem}\label{lem:CzebPerp}
If $\alpha \geq \alpha_0$ and $\varepsilon < \min\{1/2, \a_{\infty}-\a\}$ and $\rho = \rho(\alpha)$,  then  one can find a constant $c=c(\alpha,\varepsilon)>0$ such that for all $N$ and $M$,
	\begin{equation*}
		\PP[R_{N} > e^M] \leq c N^{2(\alpha+1)}  \exp \{ - M(\alpha+\varepsilon) +N\Lambda(\alpha) + N\rho\varepsilon+c N\varepsilon^2 \}.
	\end{equation*}
\end{lem}
\begin{lem}\label{lem:PrimEst} If $\a \ge \a_0$ and the assumptions of Theorem~\ref{thm:mthm1} are in force. Then for any $\eps$ such that $\a+\eps < \a_\8$ and $\EE[Z_1^{\alpha+\varepsilon}]<\infty$ one can find a constant $c=c(\alpha,\varepsilon)>0$ such that for all $N$, $M$
	\begin{equation}\label{eq:PrimEst1}
		\PP[W_{N} > e^M] \leq c N^{2(\alpha+c+1)} \exp \{ - M(\alpha+\varepsilon) +N\Lambda(\alpha) + N\rho\varepsilon+cN\varepsilon^2 \}.
	\end{equation}
\end{lem}

\begin{proof}
By Lemma \ref{lem:s4}, we have
	\begin{align*}
		\PP[W_N>e^M] & \leq \sum_{k=1}^N \PP\left[Z_k > e^M(2k^2)^{-1}\right]
		\leq \sum_{k=1}^N \EE[Z_k^{\alpha+\varepsilon}] e^{-M(\alpha+\varepsilon)}2^{\alpha+\varepsilon}k^{2(\alpha+\varepsilon)}\\
		& =  ce^{-M(\alpha+\varepsilon)}\sum_{k=1}^N e^{\Lambda(\alpha+\varepsilon)k}k^{2(\alpha+c+\varepsilon)}.
	\end{align*}
	 Since $\Lambda(\alpha+\varepsilon) \le  \Lambda(\alpha)+\rho(\alpha)\varepsilon +\sigma(\alpha) \varepsilon^2$,
	for sufficiently small $\varepsilon$ 
	\begin{align*}
		\PP[W_N>e^M] & \leq c e^{-M(\alpha+\varepsilon)}\sum_{k=1}^N e^{(\Lambda(\alpha)+\rho(\alpha)\varepsilon +\varepsilon^2\sigma(\alpha))k}k^{2(\alpha+c+\varepsilon)}\\
		& \leq c e^{-M(\alpha+\varepsilon)}e^{(\Lambda(\alpha) +\rho(\alpha)\varepsilon +\varepsilon^2 \sigma(\alpha))N}N^{2(\alpha+c+\varepsilon)+1}.
	\end{align*}
\end{proof}

\begin{cor}\label{cor:nbar}
	Assume that $n = \lfloor \rho^{-1}\log t \rfloor$ for some $\rho(\alpha_0) < \rho < \rho(\alpha_{\infty})$.
	Let $\alpha$ be chosen such
	that $\rho(\alpha) = \rho$ and pick $D > \Lambda(\alpha)^{-1}(2\alpha+2\sigma(\alpha)+3)$. We have
	\begin{equation*}
		\PP[W_{\overline{n}}>t] = o \big( (\log t)^{-1/2} t^{-\Lambda^*(\rho)/\rho}\big),
	\end{equation*}
	where $\overline{n} =\lfloor  n - D \log n\rfloor $. In partucular
$$		\PP[\max_{0\le j\le \overline n}Z_j >t] = o \big( (\log t)^{-1/2} t^{-\Lambda^*(\rho)/\rho}\big).$$

\end{cor}

\begin{proof}
	If we invoke Lemma~\ref{lem:PrimEst} with $N =\overline{n}$, $M=\log t$ and $\varepsilon = (\log t)^{-1/2}$ we infer,
	that for some constant $c$ (which is bounded by Lemma \ref{lem:s1})
	\begin{align*}
		\PP[W_{\overline{n}}>t] & \leq c t^{-\alpha-\varepsilon} \overline{n}^{2(\alpha+1)}
			\exp \{ \overline{n} \Lambda(\alpha)+ \overline{n} \rho\varepsilon + \sigma(\alpha)\overline{n} \varepsilon^2  \} \\
			& \leq c t^{-\alpha-\varepsilon} n^{2(\alpha+\sigma(\alpha)+1)}
				\exp \{ (n - D\log n) \Lambda(\alpha)+ (n - D\log n) \rho\varepsilon \} \\
			& = c t^{-\Lambda^*(\rho)/\rho} n^{2(\alpha+\sigma(\alpha)+1)}
				\exp \{ - D\log n \Lambda(\alpha) - D\log n\rho\varepsilon \} \\
			& \leq c t^{-\Lambda^*(\rho)/\rho} (\log t)^{2(\alpha+\sigma(\alpha)+1) - D \Lambda(\alpha)}.
	\end{align*}
	By the choice of $D$ the last expression is of correct order.
\end{proof}

The arguments leading to large deviation estimates for $W$ follow the same idea as the one for $Z$, namely approximation of $W_n$ by its conditional mean.
In order to be able to execute a similar procedure, denote for $n>m$
\begin{align*}
	W_{m,n} &= W_n - W_{m} = \sum_{k=m+1}^nZ_k,\\
	R_{m,n} &= \sum_{k=m+1}^n\prod_{j=m}^{k-1}A_j = A_m+A_mA_{m+1}+ \ldots + A_mA_{m+1}\ldots A_{n-1}.
\end{align*}
Then
$$
\EE\big[ W_{m,n}|Z_m, \mathcal{Q} \big] = R_{m,n-1} Z_m .
$$
In view of Corollary~\ref{cor:nbar} it suffices to investigate $\PP[W_{\overline{n},n}>t]$. For the moment, we will focus our attention on showing how to approximate $W_{\overline{n},n}$ by its conditional mean.

\begin{lem}\label{lem:diff}
For $n>m$ the following formula holds	
	\begin{equation}\label{eq:W-ZR}
 		W_{m,n} - Z_m R_{m,n-1}=  \sum_{k=m+1}^n (Z_k-A_{k-1}Z_{k-1})(1+R_{k,n}).
	\end{equation}
\end{lem}
\begin{proof}
Recall
	\begin{equation}\label{eq:z-pi}
		Z_j - \Pi_{j-1} = \sum_{k=1}^{j} (Z_k-A_{k-1}Z_{k-1}) \Pi_{k,j-1}.
	\end{equation}
	Similarly, taking a shorter telescopic sum yields, for $m <j$,
	\begin{equation}\label{eq:z-piN}
		Z_j - Z_m\Pi_{m,j-1} = \sum_{k=m+1}^{j} (Z_k-A_{k-1}Z_{k-1}) \Pi_{k,j-1}.
	\end{equation}
	If we sum the expression above over $m +1\leq j \leq n$ and change the order of summation on the right-hand side we
	will arrive at
	\begin{align*}
		W_{m,n} - Z_{m} R_{m,n-1}
			= &\sum_{j=m+1}^n\big(Z_j-Z_m\Pi_{m,j-1}\big)
			=\sum_{j=m+1}^n\sum_{k=m+1}^{j} (Z_k-A_{k-1}Z_{k-1})\Pi_{k,j-1}\\
			=&\sum_{k=m+1}^n (Z_k-A_{k-1}Z_{k-1})\sum_{j=k}^n \Pi_{k,j-1}
 			= \sum_{k=m+1}^n (Z_k-A_{k-1}Z_{k-1})(1+R_{k,n-1}),
	\end{align*} which proves the lemma.
\end{proof}

\begin{lem} \label{lem:5.7}
	Under the assumptions of Theorem~\ref{thm:mthm1} fix $\rho = \rho(\a)$ for some $\a\in (\a_0,\a_\8)$ and assume additionally \eqref{eq:h1} pick
	$D > \Lambda(\alpha)^{-1}(2\alpha+7)$. We have
	\begin{equation*}
		\PP[|W_{\overline{n},n} -   Z_{n'}\Pi_{n',\overline{n}-1} R_{\overline{n},n-1}| >t]
			= o \left( (\log t)^{-1/2} t^{-\Lambda^*(\rho)/\rho} \right),
	\end{equation*}
	where $\overline{n} = n - D \log n$ and $n' = K \log n$ with  $K$
	large enough.
\end{lem}
\begin{proof} First we compare $W_{\overline{n},n}$ with   $Z_{\overline{n}} R_{\overline{n},n-1}$. Fix $\varepsilon_0>0$.
	Using similar argument as in the proof of Lemma~\ref{lem:s4}, one can show that there are $c>0$ and $\delta \in (0,1)$
	\begin{equation*}
		\EE[|Z_k-A_{k-1}Z_{k-1} |^s] \leq c \delta^k \lambda(s)^k, \quad s \in (\alpha, \alpha+\varepsilon_0).
	\end{equation*}
	Combining this with  Lemma \ref{lem:diff} and Lemma~\ref{lem:CzebPerp} we obtain there exist $\delta>0$ such that for $\varepsilon \in (0, \varepsilon_0)$,
	\begin{align*}
		&\PP\Big[\big|W_{\overline{n},n} - Z_{\overline{n}} R_{\overline{n},n-1 }\big| >t\Big]  \leq
		\sum_{k=\overline{n}+1}^n \PP \Big[ |Z_k-A_{k-1}Z_{k-1}|(1+R_{k,n-1}) > t\big(2(k-\overline n)^2\big)^{-1}\Big]\\
		& \leq  c \sum_{k=\overline{n}+1}^n\EE\big[|Z_k-A_{k-1}Z_{k-1} |^{\alpha+\varepsilon}\big] t^{-\alpha-\varepsilon}
			(k-\overline n)^{2(\alpha+1)}  \lambda(\alpha)^{n-k} (n-k)^{2(\alpha+1)}
			e^{(n-k)( \rho\varepsilon+c \varepsilon^2) }\\
		& \leq c t^{-\overline {\alpha}} (\log n)^{4(\alpha+2)+1} n^{D(\rho \varepsilon + c\varepsilon^2)} \delta^{\overline n}
=o \big( (\log t)^{-1/2} t^{-\Lambda^*(\rho)/\rho}\big)
	\end{align*}
	provided that $K$ is chosen large enough. Here we used a version of the second claim in Lemma~\ref{lem:s4} which is locally uniform in $\alpha$. This can be proven using similar methods.   Next, by Lemma \ref{lem:diff} we may write
	\begin{equation*}
		Z_{\overline{n}} - Z_{n'}\Pi_{n',\overline{n}-1} = \sum_{k=n'+1}^{\overline{n}} (Z_k - A_{k-1}Z_{k-1}) \Pi_{k, \overline{n}-1}.
	\end{equation*}
	Since by Lemma~\ref{lem:CzebPerp}, for $n'+1\leq k \leq \overline{n}$ and  $\varepsilon = (\log t)^{-1/2}$ with $t$
	big enough
	\begin{align*}
\PP\Big[ |Z_{\overline n} - &Z_{n'}\Pi_{n',\overline n-1}| R_{\overline n, n-1} >t\Big]
\le \sum_{k=n'+1}^{\overline n} \PP\Big[ |Z_k - A_{k-1}Z_{k-1}|  \Pi_{k,\overline n-1} R_{\overline n,n-1} > t \big( 2 (k-n')\big)^{-1} \Big]\\
&\le c \sum_{k=n'+1}^{\overline n} \EE \big[|Z_k - A_{k-1}Z_{k-1}|^{\alpha+\varepsilon}\big] \lambda(\alpha+\varepsilon)^{\overline{n}-k}
(\log n)^{2(\alpha+1)} t^{-\alpha - \varepsilon} (k-n')^{2(\alpha+1)}e^{(n-\overline n) (\Lambda(\alpha) + \rho \varepsilon + c \varepsilon^2)}\\
&\le c t^{-\Lambda^*(\rho)/\rho} e^{cn\varepsilon^2} \cdot e^{\overline{n}(\Lambda(\alpha+\varepsilon) - \Lambda(\alpha))}e^{-\overline{n}\varepsilon \rho}e^{-\overline{n}c \varepsilon^2}\cdot \sum_{k=n'+1}^{\overline{n}} \delta^k (k-n')^{2(\alpha+1)}\\
&\le c t^{-\Lambda^*(\rho)/\rho} \rho^{n'} = C (\log t)^{K \log \delta} t^{-\Lambda^*(\rho)/\rho}
=o \big( (\log t)^{-1/2} t^{-\Lambda^*(\rho)/\rho}\big).
	\end{align*}
for appropriately large $K$.
\end{proof}

\begin{proof}[Proof of Theorem~\ref{thm:mthmld}]
	By previous considerations, we only need to consider
	\begin{equation*}
		\PP\left[ Z_{n'}\Pi_{n',\overline{n}-1}(1+R_{\overline{n},n}) > t \right].
	\end{equation*}
	Denote $H_{j_n} = Z_{n'}(1+R_{\overline{n},n})$ and apply Lemma~\ref{lem:branchPetrov} to infer that
	\begin{equation*}
	\PP\left[ Z_{n'}\Pi_{n',\overline{n}-1}(1+R_{\overline{n},n}) > t \right] = \PP\left[H_{j_n}\Pi_{j_n,n} >t\right] \sim
	\frac{\mathcal{C}_5(\alpha)}{\sqrt{\log t}} t^{-\Lambda^*(\rho)/\rho},
\end{equation*}
	where	
	$
		\mathcal{C}_5(\alpha) =\frac{c_Rc_Z}{\alpha \sigma(\alpha) \sqrt{2\pi }}$.

The second claim of Theorem~\ref{thm:mthmld} is an immediate consequence of  Lemma~\ref{lem:PrimEst}.
\end{proof}

\subsection{ Limit Theorems of $T_t^W$}\label{sec:5.2}

We start with the following Lemma. 
\begin{lem}\label{lem:m1}
Assume $\rho = \rho_0$ and let $n_1 = n-b\sqrt{n\log n}$ and $n_2 = n + b\sqrt{n\log n}$. Then for any $\delta>0$ one can pick $b>0$ large enough such that
$$\PP[W_{n_1} > t]\le C t^{-\alpha_0} (\log t)^{-\delta}$$
and
$$
\PP[ W_{n_2,\infty} > t] \le C t^{-\alpha_0} (\log t)^{-\delta},
$$
where $W_{n_2, \infty} = \sum_{j=n_2+1}^{\infty}Z_j$.
\end{lem}
\begin{proof}
Applying Lemma \ref{lem:PrimEst} and choosing $\varepsilon = \sqrt{\log n / n}$ we obtain
\begin{align*}
\P[W_{n_1} > t] &\le n_1^{2(\alpha_0+1)} t^{-\alpha_0-\eps} e^{n_1 \rho \eps + c n_1 \eps^2}
\le t^{-\alpha_0} n^{2(\alpha_0+1)}e^{-b\log n} e^{c\log n}\\
& = t^{-\alpha_0} n^{2(\a_0+1) +c-b}
\le C t^{-\a_0} (\log t)^{-\delta}
\end{align*}
for appropriately large $b$.

To prove the second part of the Lemma we proceed similarly as above, but this time $\eps$ depends also on the parameter $k$: $\eps = \eps(k,n)>0$. We estimate
\begin{align*}
\P[W_{n_2,\8} > t] &\le \sum_{k=n_2+1}^\8 \P\bigg[ Z_k > \frac{t}{2(k-n_2)^2}\bigg]\\
&\le C \sum_{k=n_2+1}^\8 \E\big[  Z_k^{\a_0-\eps(k,n)} \big] t^{-\a_0 + \eps(k,n)} (k-n_2)^{2\a_0}\\
&\le C t^{-\a_0} \sum_{k=n_2    +1}^\8 t^{\eps(k,n)} \l(\a_0-\eps(k,n))^k (k-n_2)^{2\a_0}.
\end{align*}
Now for some large $N$ we  consider separately two cases when $k\le Nn$ and $k>Nn$.
First we consider large values of $k$ and then  we just choose $\eps(k,n)=\eps_2$ for some small fixed $\eps_2$. Let
$\gamma = \l(\a_0-\eps_2)<1$.
\begin{align*}
t^{-\a_0} \sum_{k=Nn}^\8 t^{\eps(k,n)} \l(\a_0-\eps(k,n))^k (k-n_2)^{2\a_0}
&\le t^{-\a_0}  t^{\eps_2 } \sum_{k>Nn} \gamma^k k^{2\a} \\
&\le t^{-\a_0}  t^{\eps_2 } \gamma^{\frac{nN}2} \sum_{k>Nn} \gamma^{k/2} k^{2\a} \\
&\le C t^{-\a_0}  t^{-\delta}
\end{align*}
for appropriately large $N$.

 In the second case choose $\eps(k,n) = \eps_1 = \sqrt{\log n/n}$ and recall $\l(\a_0-\eps) \le e^{-\eps\rho_0 + c\eps^2}$. Then
\begin{align*}
t^{-\a_0} \sum_{k=n_2+1}^{Nn} t^{\eps(k,n)} \l(\a_0-\eps(k,n))^k &(k-n_2)^{2\a_0}\\
&\le t^{-\a_0}  e^{\rho_0 \sqrt{n \log n}} \sum_{n_2 \le k \le Nn}   e^{ - \rho_0 k\sqrt{\log n/n}}   e^{C k \log n/ n}
 (k-n_2)^{2\a_0}\\
&\le t^{-\a_0}  e^{\rho_0 \sqrt{n \log n}} \cdot Nn \cdot    e^{-  \rho_0 n_2 \sqrt{\log n/n}}   e^{C N \log n/ n}
 (Nn)^{2\a_0}\\
&\le C_N  t^{-\a_0}  e^{- \rho_0 b \log n } n^{CN + 2\a_0+1}\\
&\le C_N t^{-\a_0} n^{-\rho_0 b + CN + 2\a_0+1}\\
&\le C_N t^{-\a_0} (\log t)^{-\delta}
\end{align*}
for large $b$.
\end{proof}

\begin{lem}\label{lem:m2}
For $\rho=\rho_0$ we have
$$
\P\Big[ \big| W_{n_y} - Z_{n'} \Pi_{n',n_1-1} R_{n_1,n_y-1} \big| > t
\Big] = o(t^{-\a_0}) \qquad t\to\8,
$$ where $n_y = n_1 + c_0y\sqrt{\log n}$, $n_1$ as in Lemma \ref{lem:m1} and $n'= K\log n$ for large $K$ and $c_0 = \sigma_0\rho_0^{-3/2}$.
\end{lem}
This Lemma can be proved exactly in the same way as Lemma \ref{lem:5.7}. We left details for the reader.
\begin{proof}[Proof of Theorem \ref{thm:mthmclt}]
{\sc Step 1. Law of large numbers.} The SLLN is a direct consequence of Lemma \ref{lem:m1} and \eqref{eq:af-kes}. Indeed, for any $\eps>0$, we have
\begin{align*}
  \P\bigg[\bigg| \frac{T_t^W}{\log t} - \frac 1{\rho_0} \bigg| > \eps \; & \bigg|\; T_t^W <\8 \bigg] \\
  & \le \P\Big[ T_t^W < \log t\big(1/\rho_0 -\eps \big)\; \Big| \; T_t^W<\8   \Big]
  + \P\Big[ T_t^W > \log t\big(1/\rho_0 +\eps \big)\; \Big| \; T_t^W<\8   \Big]\\
  & \le \P\big[ W_{n_1} >t \; \big|\; T_t^W<\8   \Big]
+ \P\big[ W_{n_2,\8} >t \; \big|\;  T_t^W<\8   \Big]\\
  & \le C(\log t)^{-\delta}.
\end{align*}

\noindent
{\sc Step 2. Central limit theorem.} The second part of the
 Theorem can be proved using similar arguments as in the proof of Theorem 2.2 in \cite{buraczewski2014large}. However in our case some additional problems arise. Thus we focus here on the main arguments, emphasising the differences. We refer the reader to \cite{buraczewski2014large} for all the details.

\medskip

\noindent
{\sc Step 2a. Petrov's result.} The result follows essentially from Petrov's Theorem (Lemma \ref{lem: petrov1}) and first we explain how it should be applied. In view of Lemma \ref{lem:m2}  we need prove that
$$
\lim_{t\to\8} t^{\a_0 }\P\Big[ Z_{n'} \Pi_{n',n_1-1}R_{n_1, n_y-1}  > t \Big] = C_1 c(\a_0) \Phi(y).
$$
For 'fixed' $Z_{n'} R_{n_1,n_y-1}$ we want to apply Lemma   \ref{lem: petrov1}. This can be done only for some restricted set of values. The details are as follows. Let $\sigma_0=\sigma(\alpha_0)$ and
\begin{align*}
I(t) &= \big[ 0, \rho(n-n_1+n') + (y+D)\sigma_0 \sqrt{n_1-n'} \big]\\
V_n &= Z_{n'} R_{n_1,n_y-1}.
\end{align*}
Below we apply Lemma \ref{lem: petrov1} with $(n,t,\delta_n)$ replaced by $(n_1-n', e^{\rho_0(n_1-n')},
\frac{\rho_0(n-n_1+n')-s}{n_1-n'})$. Let $F_t$ be the distribution function of  $\log V_n$ (recall that $n$ depends on $t$), then
\begin{align*}
\P\big[ V_n \Pi_{n',n_1-1} > t, V_n \in I(t) \big]
&= \int_{ I(t)} \P\big[ \Pi_{n',n_1-1} > t e^{-s} \big] dF_t(s)\\
&= \frac{1+o(1)}{\alpha_0\sigma_0 \sqrt{2\pi (n_1 - n')}} \int_{ I(t)} t^{-\a_0} e^{\a_0 s}
e^{-\frac{(n_1-n')\delta_n}{2\sigma_0^2}} dF_t(s)\\
&= \frac{(1+o(1)) t^{-\a_0}}{ \alpha_0\sigma_0 \sqrt{2\pi (n_1 - n')}} \int_{ I(t)}  e^{\a_0 s}
e^{-\frac{(\rho(n-n_1+n')-s)^2}{2\sigma_0^2(n_1-n')}} dF_t(s).
\end{align*}
Next we change variables applying the transformation
$$
T_t(s) = \frac{s- \rho(n-n_1+n')}{\sigma_0\sqrt{n_1-n'}}
$$ and defining the  distribution $G_t = F_t \circ T_t^{-1}$ we obtain
\begin{equation}\label{eq:m8}
\P\big[ V_n \Pi_{n',n_y-1} > t, V_n \in I(t) \big]
= \frac{(1+o(1)) t^{-\a_0}}{\alpha_0 \sigma_0 \sqrt{2\pi (n_1 - n')}} \int_{-\frac{ \rho(n-n_1+n')}{\sigma_0\sqrt{n_1-n'}} }^y   e^{\a_0 T_t^{-1}(u)}
e^{-u^2/2} dG_t(u).
\end{equation}

\medskip

\noindent
{\sc Step 2b. Uniform convergence.}
To proceed further we need a technical  observation that for $-\8 < a< b< y$
\begin{equation}
\label{eq:m4}
\lim_{t\to \8 } e^{\a_0 s} \overline F_t(s) = C_1 c_1(\a_0) \ \mbox{uniformly for  }  s\in T_t^{-1}([a,b]),
\end{equation}
where $C_1$ is the Kesten-Goldie constant. More precisely, let $R_\8 =\sum_{k=1}^\8 A_1\ldots A_k$, then under assumptions of Theorem \ref{thm:mthmclt}  Kesten
\cite{kesten:1973} and Goldie \cite{goldie:1991} proved
\begin{equation}\label{eq:msss}
  \lim_{t\to\8} t^{\a_0} \P[R_\8 > t] = C_1 > 0.
\end{equation}

To prove the above statement we apply (4.35) from \cite{buraczewski2014large}
\begin{equation}
\label{eq:ms}
\lim_{t\to \8 } e^{\a_0 s} \P\big[ R_{n_1,n_y} > e^s \big]  \to C_1  \ \mbox{uniformly for  }  s\in T_t^{-1}([a,b])
\end{equation}
Then
\begin{align*}
  e^{\a_0 s} \overline F_t (s) & = e^{\a_0 s} \P\big[ Z_{n'} R_{n_1,n_y} > e^s \big]\\
  &= e^{\a_0 s } \sum_{k=1}^\8 \P\big[ R_{n_1, n_y} > e^s/k \big] \P[Z_{n'}=k]\\
  &= e^{\a_0 s } \sum_{k\le e^{\delta \sqrt n}} \P\big[ R_{n_1, n_y} > e^s/k \big] \P[Z_{n'}=k]
  + e^{\a_0 s } \sum_{k> e^{\delta \sqrt n}} \P\big[ R_{n_1, n_y} > e^s/k \big] \P[Z_{n'}=k]\\
  &= I + II.
\end{align*}
We will prove that the first term gives the asymptotic behaviour and the second one is negligible. To estimate the latter,  by the H\"older inequality and \eqref{eq:msss}, we may write
\begin{equation}\label{eq:mssss}
\begin{split}
  II & \le \sum_{k\ge e^{\delta \sqrt n}} k^{\a_0} \P[Z_{n'}=k] = \E\Big[ Z_{n'}^\a {\bf 1}_{\{ Z_{n'}> e^{\delta \sqrt n}\}}
  \Big]\\
  &\le \E\big[ Z_{n'}^{p\a} \big]^{1/p} \P\big[ Z_{n'} > e^{\delta \sqrt n} \big]^{1/q}\\
  &\le C e^{\Lambda(p\a) K \log n/p} e^{-\delta \a_0 \sqrt n/q} \E\big[ Z_{n'}^{\a_0}\big]^{1/q} = o(1)
\end{split}
\end{equation}
For $k\le e^{\delta \sqrt n}$, $s-\log k\in T^{-1}_t([a-2\delta,b])$ and by \eqref{eq:ms} we have that for small $\eps$ and large $n$
$$
(C_1 -\eps) \E\big[ Z_{n'}^\a {\bf 1}_{\{ k\le e^{\delta\sqrt n }\}}  \big] \le I \le
(C_1 +\eps) \E\big[ Z_{n'}^\a {\bf 1}_{\{ k\le e^{\delta\sqrt n }\}}  \big].
$$ Applying \eqref{eq:mssss} ,  Lemma~\ref{lem:s4}, passing first with $n\to\8$ and then with $\eps\to 0$, we obtain
$$
\lim_{s\to\8} I = C_1 c(\a_0).
$$

\medskip

\noindent
{\sc Step 2c. Convergence to the Lebesgue measure.} Now our aim is to prove that for any $f\in C_C(-\8,y)$ (continuous, compactly supported function in $(-\8,y)$):
\begin{equation}
\label{eq:m6}
\lim_{t\to\8} \int_{-\8}^y f(u) dH_t(u) = C_1 c(\a_0) \int_{-\8}^y f(u)du,
\end{equation}  where $dH_t(u) = \frac{e^{\a_0 T_t^{-1}(u)}}{\a_0 \sigma_0 \sqrt{n_1-n'}}dG_t(u)$
and
\begin{equation}\label{eq:m7}
  H_t(v,w) \le C (w-v) + \frac C{\sqrt{n_1}}
\end{equation} for $-\8 < v < w< y$ and some constant $C$.

Fix $-\8<v<w<y$, $v^*(t) = T_t^{-1}(v)$, $w^*(t) = T_t^{-1}(w)$, then integrating by parts
\begin{multline*}
H_t(v,w)\\ = -\frac{1}{\a_0\sigma_0\sqrt{n_1-n'}} \big(
e^{\a_0 w^*(t)}  \overline F_t(w^*(t))
 - e^{\a_0 v^*(t)}  \overline F_t(v^*(t))
\big) + \frac{1}{\sigma_0\sqrt{n_1-n'}} \int_{v^*(t)}^{w^*(t)} e^{\a_0 u} \overline F_t (u)du.
\end{multline*}
and \eqref{eq:m4} implies \eqref{eq:m7}.

To prove \eqref{eq:m6} observe that the first term above is negligible and  we again apply \eqref{eq:m4} and obtain
$$
\lim_{t\to\8} H_t(v,w) =  \lim_{t\to \8} \frac{1}{\sigma_0 \sqrt{n_1-n'}} \int_{v^*(t)}^{w^*{t}} e^{\a_0 u} \overline F_t(u)du
= \frac{C_1 c(\a_0) (w^*(t) - v^*(t))}{\sigma_0\sqrt{n_1-n'}} = C_1 c(\a_0)(w-v).
$$
Finally applying the standard procedure and approximating an arbitrary $f$ by Riemann sums we obtain \eqref{eq:m6}.

\medskip

\noindent
{\sc Step 2d. Conclusion.} Now we are able to conclude. For large $N$ and small $\delta$ we split the integral \eqref{eq:m8} into three parts:
$(-\frac{ \rho(n-n_1+n')}{\sigma_0\sqrt{n_1-n'}} ,-N]$, $(-N, y-\delta]$, $(y-\delta, y)$,
and in view of \eqref{eq:m6} and \eqref{eq:m7} we have
\begin{align*}
  \lim_{t\to\8} \int_{-N}^{y-\delta} e^{-u^2/2} dH_t(u) & = C_1 c(\a_0) \int_{-N}^{y-\delta} e^{-u^2/2}du,\\
  \lim_{t\to\8} \int^{-N}_{-\frac{ \rho(n-n_1+n')}{\sigma_0\sqrt{n_1-n'}} } e^{-u^2/2} dH_t(u) & \le e^{-N^2/2},\\
  \lim_{t\to\8} \int_{y-\delta}^{y} e^{-u^2/2} dH_t(u) & \le C\delta.
\end{align*}
Passing with $N\to\8$ and $\delta\to 0$ we obtain
$$
\lim_{t\to\8} t^{\a_0} \P\big[ V_n \Pi_{n',n_y-1} > t, V_n \in I(t) \big] = \frac{C_1 c(\a_0)}{\sqrt{2\pi}} \int_{-\8}^y
e^{-u^2/2} du = \frac{C_1 c(\a_0)}{\sqrt{2\pi}} \Phi(y)
$$
\medskip

\noindent
{\sc Step 2e. The negligible part.} To complete the proof we need to justify that the remaining part is negligible, i.e.
$$
\lim_{t\to\8} t^{\a_0} \P\big[ V_n \Pi_{n',n_y-1} > t, V_n \notin I(t) \big] = 0.
$$ However we omit the arguments here and refer to \cite{buraczewski2014large} (proof of Theorem 2, step 4) for more details.
\end{proof}

\nocite{*}
\bibliographystyle{plain}
\bibliography{dbura_pdysz_PLDEforBPRE_bib}

\begin{thebibliography}{10}

\bibitem{afanasyev2001maximum}
V.~I. Afanasyev.
\newblock On the maximum of a subcritical branching process in a random
  environment.
\newblock {\em Stochastic processes and their applications}, 93(1):87--107,
  2001.

\bibitem{afanasyev2013high}
V.~I. Afanasyev.
\newblock High level subcritical branching processes in a random environment.
\newblock {\em Proceedings of the Steklov Institute of Mathematics},
  282(1):4--14, 2013.

\bibitem{afanaseev:2014}
V.~I. Afanasyev.
\newblock Functional limit theorems for high-level subcritical branching
  processes in a random environment.
\newblock {\em Diskret. Mat.}, 26(2):6--24, 2014.

\bibitem{Afanasyev:Boinghoff:Kersting:Vatutin:2012}
V.~I. Afanasyev, C.~B\"oinghoff, G.~Kersting, and V.~A. Vatutin.
\newblock Limit theorems for weakly subcritical branching processes in random
  environment.
\newblock {\em J. Theoret. Probab.}, 25(3):703--732, 2012.

\bibitem{Afanasyev:Boinghoff:Kersting:Vatutin:2014}
V.~I. Afanasyev, Ch. B\"oinghoff, G.~Kersting, and V.~A. Vatutin.
\newblock Conditional limit theorems for intermediately subcritical branching
  processes in random environment.
\newblock {\em Ann. Inst. Henri Poincar\'e Probab. Stat.}, 50(2):602--627,
  2014.

\bibitem{Afanasyev:Geiger:Kersting:Vatutin:2005}
V.~I. Afanasyev, J.~Geiger, G.~Kersting, and V.~A. Vatutin.
\newblock Criticality for branching processes in random environment.
\newblock {\em Ann. Probab.}, 33(2):645--673, 2005.

\bibitem{athreya2012branching}
K.~B. Athreya and P.~E. Ney.
\newblock {\em Branching processes}, volume 196.
\newblock Springer Berlin Heidelberg, 1972.

\bibitem{bahadur:rao}
R.~R. Bahadur and R.~Ranga~Rao.
\newblock On deviations of the sample mean.
\newblock {\em Ann. Math. Statist.}, 31:1015--1027, 1960.

\bibitem{bansaye:berestycki}
V.~Bansaye and J.~Berestycki.
\newblock Large deviations for branching processes in random environment.
\newblock {\em Markov Process. Related Fields}, 15(4):493--524, 2009.

\bibitem{Bansaye:Boinghoff:2011}
V.~Bansaye and C.~B\"oinghoff.
\newblock Upper large deviations for branching processes in random environment
  with heavy tails.
\newblock {\em Electron. J. Probab.}, 16:no. 69, 1900--1933, 2011.

\bibitem{Birkner:Geiger:Kersting:2005}
M.~Birkner, J.~Geiger, and G.~Kersting.
\newblock Branching processes in random environment---a view on critical and
  subcritical cases.
\newblock In {\em Interacting stochastic systems}, pages 269--291. Springer,
  Berlin, 2005.

\bibitem{Boinghoff:Kersting:2010}
C.~B\"oinghoff and G.~Kersting.
\newblock Upper large deviations of branching processes in a random
  environment---offspring distributions with geometrically bounded tails.
\newblock {\em Stochastic Process. Appl.}, 120(10):2064--2077, 2010.

\bibitem{buraczewski2014large}
D.~Buraczewski, J.~F. Collamore, E.~Damek, and J.~Zienkiewicz.
\newblock Large deviation estimates for exceedance times of perpetuity
  sequences and their dual processes.
\newblock {\em Ann. Probab.}, 44(6):3688--3739, 2016.

\bibitem{buraczewski:damek:mikosch}
D.~Buraczewski, E.~Damek, and T.~Mikosch.
\newblock {\em Stochastic models with power-law tails. The equation $X=AX+B$}.
\newblock Springer Series in Operations Research and Financial Engineering.
  Springer, [Cham], 2016.

\bibitem{buraczewski2015pointwise}
D.~Buraczewski, E.~Damek, and J.~Zienkiewicz.
\newblock Pointwise estimates for exceedance times of perpetuity sequences.
\newblock {\em Stochastic Processes and their Applications}, 128(9):2923--2951,
  2018.

\bibitem{buraczewski:dyszewski}
D.~Buraczewski and P.~Dyszewski.
\newblock Precise large deviations for random walks in random environment.
\newblock {\em Electronic Journal of Probability}, 23(114), 2018.

\bibitem{buraczewski2016precise}
D.~Buraczewski and M.~Ma{\'s}lanka.
\newblock Precise large deviations of the first passage time.
\newblock {\em to appear in Proc. AMS}, 2019.

\bibitem{dembo1996tail}
A.~Dembo, Y.~Peres, and O.~Zeitouni.
\newblock Tail estimates for one-dimensional random walk in random environment.
\newblock {\em Communications in Mathematical Physics}, 181(3):667--683, 1996.

\bibitem{DZ}
A.~Dembo and O.~Zeitouni.
\newblock {\em Large deviations techniques and applications}.
\newblock Jones and Bartlett Publishers, Boston, MA, 1993.

\bibitem{geiger2003limit}
J.~Geiger, G.~Kersting, and V.~A. Vatutin.
\newblock Limit theorems for subcritical branching processes in random
  environment.
\newblock In {\em Annales de l'Institut Henri Poincar{\'e} (B) Probabilit{\'e}s
  et Statistiques}, volume~39, pages 593--620, 2003.

\bibitem{goldie:1991}
C.~M. Goldie.
\newblock Implicit renewal theory and tails of solutions of random equations.
\newblock {\em Ann. Appl. Probab.}, 1(1):126--166, 1991.

\bibitem{grama:liu:miqueu}
I.~Grama, Q.~Liu, and E.~Miqueu.
\newblock Berry-{E}sseen's bound and {C}ramer's large deviation expansion for a
  supercritical branching process in a random environment.
\newblock {\em Stochastic Process. Appl.}, 127:1255--1281, 2017.

\bibitem{guivarc2001proprietes}
Y.~Guivarc'h and Q.~Liu.
\newblock Propri{\'e}t{\'e}s asymptotiques des processus de branchement en
  environnement al{\'e}atoire.
\newblock {\em Comptes Rendus de l'Acad{\'e}mie des Sciences-Series
  I-Mathematics}, 332(4):339--344, 2001.

\bibitem{gut2009stopped}
A.~Gut.
\newblock {\em Stopped random walks}.
\newblock Springer, 2009.

\bibitem{haccou2005branching}
P.~Haccou, P.~Jagers, and V.~A. Vatutin.
\newblock {\em Branching processes: variation, growth, and extinction of
  populations}.
\newblock Number~5. Cambridge {U}niversity {P}ress, 2005.

\bibitem{huang2012moments}
Ch. Huang and Q.~Liu.
\newblock Moments, moderate and large deviations for a branching process in a
  random environment.
\newblock {\em Stochastic Processes and their Applications}, 122(2):522--545,
  2012.

\bibitem{kersting2017discrete}
G.~Kersting and V.~Vatutin.
\newblock {\em Discrete time branching processes in random environment}.
\newblock John Wiley \& Sons, 2017.

\bibitem{kesten:1973}
H.~Kesten.
\newblock Random difference equations and renewal theory for products of random
  matrices.
\newblock {\em Acta Math.}, 131:207--248, 1973.

\bibitem{kesten1975limit}
H.~Kesten, M.~V. Kozlov, and F.~Spitzer.
\newblock A limit law for random walk in a random environment.
\newblock {\em Compositio Mathematica}, 30(2):145--168, 1975.

\bibitem{kozlov:2006}
M.~V. Kozlov.
\newblock On large deviations of branching processes in a random environment: a
  geometric distribution of the number of descendants.
\newblock {\em Diskret. Mat.}, 18(2):29--47, 2006.

\bibitem{kozlov:2010}
M.~V. Kozlov.
\newblock On large deviations of strictly subcritical branching processes in a
  random environment with a geometric distribution of descendants.
\newblock {\em Teor. Veroyatn. Primen.}, 54(3):439--465, 2009.

\bibitem{Lalley}
S.~P. Lalley.
\newblock Limit theorems for first-passage times in linear and nonlinear
  renewal theory.
\newblock {\em Adv. in Appl. Probab.}, 16(4):766--803, 1984.

\bibitem{petrov1965probabilities}
V.~V. Petrov.
\newblock On the probabilities of large deviations for sums of independent
  random variables.
\newblock {\em Theory of Probability \& Its Applications}, 10(2):287--298,
  1965.

\bibitem{smith1969branching}
W.~L. Smith and W.~E. Wilkinson.
\newblock On branching processes in random environments.
\newblock {\em The Annals of Mathematical Statistics}, pages 814--827, 1969.

\bibitem{tanny1}
D.~Tanny.
\newblock Limit theorems for branching processes in a random environment.
\newblock {\em Ann. Probability}, 5(1):100--116, 1977.

\bibitem{tanny1988necessary}
D.~Tanny.
\newblock A necessary and sufficient condition for a branching process in a
  random environment to grow like the product of its means.
\newblock {\em Stochastic Processes and their Applications}, 28(1):123--139,
  1988.

\end{thebibliography}
\end{document}